\documentclass[12pt]{amsart}
\usepackage{txfonts}
\usepackage{bbm}
\usepackage{mathrsfs}
\usepackage{amssymb}

\usepackage{amssymb,amsmath,amsthm,amscd}

\usepackage[all]{xy}

\numberwithin{equation}{section}

\newtheoremstyle{fancy1}{10pt}{10pt}{\itshape}{12pt}{\textsc\bgroup}{.\egroup}{8pt}{}
\newtheoremstyle{fancy2}{10pt}{10pt}{}{12pt}{\itshape}{.}{8pt}{ }

\theoremstyle{fancy1}

\newtheorem{cor}[equation]{Corollary}
\newtheorem{lem}[equation]{Lemma}

\newtheorem{prop}[equation]{Proposition}
\newtheorem{thm}[equation]{Theorem}

\newtheorem{main}{Theorem}
\newtheorem*{main*}{Theorem}

\newtheorem*{cor*}{Corollary}

\theoremstyle{fancy2}

\newtheorem{rem}[equation]{Remark}
\newtheorem*{rem*}{Remark}

\newcommand{\cref}[1]{Corollary~\ref{#1}}

\newcommand{\Sph}{\mathbb{S}}

\newcommand{\W}{\mathsf{W}}
\newcommand{\Q}{\mathsf{Q}}

\newcommand{\N}{\mathsf{N}}

\newcommand{\CP}{\mathbb{C\mkern1mu P}}

\newcommand{\C}{{\mathbb{C}}}

\newcommand{\Z}{{\mathbb{Z}}}

\renewcommand{\H}{\ensuremath{\operatorname{\mathsf{H}}}}
\newcommand{\E}{\ensuremath{\operatorname{\mathsf{E}}}}
\newcommand{\F}{\ensuremath{\operatorname{\mathsf{F}}}}
\newcommand{\G}{\ensuremath{\operatorname{\mathsf{G}}}}

\newcommand{\SO}{\ensuremath{\operatorname{\mathsf{SO}}}}
\renewcommand{\O}{\ensuremath{\operatorname{\mathsf{O}}}}
\newcommand{\Sp}{\ensuremath{\operatorname{\mathsf{Sp}}}}
\newcommand{\U}{\ensuremath{\operatorname{\mathsf{U}}}}
\newcommand{\SU}{\ensuremath{\operatorname{\mathsf{SU}}}}
\newcommand{\PSU}{\ensuremath{\operatorname{\mathsf{PSU}}}}
\newcommand{\Spin}{\ensuremath{\operatorname{\mathsf{Spin}}}}

\newcommand{\T}{\ensuremath{\operatorname{\mathsf{T}}}}
\renewcommand{\S}{\ensuremath{\operatorname{\mathsf{S}}}}

\newcommand{\M}{\ensuremath{\operatorname{\mathsf{M}}}}
\newcommand{\A}{\ensuremath{\operatorname{\mathsf{A}}}}
\newcommand{\CC}{\ensuremath{\operatorname{\mathsf{C}}}}

\newcommand{\g}{\ensuremath{\operatorname{\mathsf{g}}}}

\def\con#1=#2(#3){#1 \equiv #2 \bmod{#3}}

\newcommand{\no}{\noindent}

\newcommand{\K}{\mathsf{K}}
\renewcommand{\L}{\mathsf{L}}

\renewcommand{\F}{\mathsf{F}}

\begin{document}


\title{Rank three geometry and positive curvature}

\author{Fuquan Fang}
\address{Department of Mathematics\\
Capital Normal University\\
Beijing 100048
\\China}
\email{fuquan\_fang@yahoo.com }

\author{Karsten Grove}
\address{Department of Mathematics\\
University of Notre Dame\\
      Notre Dame, IN 46556\\USA
      }
\email{kgrove2@nd.edu}

\author{Gudlaugur Thorbergsson}
\address{Mathematisches Institut\\
Universit\"at zu K\"oln\\
Weyertal 86-90\\
50931 K\"oln\\ Germany
      }
\email{gthorber@mi.uni-koeln.de }

\thanks{The first author is supported in part by an NSFC grant and he is grateful to the University of Notre Dame for its hospitality.
The second author is supported in part by an NSF grant, a research chair at the Hausdorff Center at the University of Bonn and by a Humboldt research award. The third author is grateful to the University
of Notre Dame and the Capital Normal University in Beijing for their
hospitality. }

\begin{abstract}
An axiomatic characterization of buildings of type $\CC_3$ due to Tits is used to prove that any cohomogeneity two polar action of type $\CC_3$ on a positively curved simply connected manifold is equivariantly diffeomorphic to a polar action on a rank one symmetric space. This includes two actions on the Cayley plane whose associated $\CC_3$ type geometry is not covered by a building.
\end{abstract}

\maketitle

 The rank (or size) of a Coxeter matrix $\M$ coincides with the number of generators of its associated Coxeter system. The basic objects in Tits' local approach to buildings \cite{Ti2} are the so-called chamber systems $\mathscr{C}$ of type $\M$ (see also \cite{Ro}). Indeed, if any so-called (spherical) residue (subchamber system) of $\mathscr{C}$ of rank 3 is covered by a building, so is $\mathscr{C}$.
 
 Recall that a polar $\G$ action on a Riemannian manifold $M$ is an isometric action with a so-called section $\Sigma$, i.e., an immersed submanifold of $M$ that meets all $\G$ orbits orthogonally. Since the action by the identity component of $\G$ is polar as well, we assume throughout without stating it that $\G$ is connected.
 
 It is a key observation of \cite{FGT} that the study of polar $\G$ actions on 1-connected positively curved manifolds $M$ in essence is the study of a certain class of (connected) chamber systems $\mathscr{C}(M;\G)$. Moreover, when the universal (Tits) cover of $\mathscr{C}(M;\G)$ is a building it has the structure of a compact spherical building in the sense of Burns and Spatzier \cite{BSp}. This was utilized in \cite{FGT} to show:
 
 \begin{main}
 Any polar $\G$ action of cohomogeneity at least two on a simply connected closed positively curved manifold $M$ is equivariantly diffeomorphic to a polar $\G$ action on a rank one symmetric space if the associated chamber system $\mathscr{C}(M;\G)$ is not of type $\CC_3$.
  \end{main}
 
 We note here, that when the action has no fixed points, the rank of $\mathscr{C}(M;\G)$ is $\dim(M/\G) +1$,
 i.e., \emph{one more} than the \emph{cohomogeneity} of the action. In the above theorem the Cayley plane emerges only in cohomogeneity two and when $\G$ has fixed points. Moreover, there are indeed chamber systems with type $\M = \CC_3$ whose universal cover is NOT a building (see, e.g., \cite{Ne}, \cite{FGT}, \cite{Ly}, \cite{KL}
 and below). In our case, a polar $\G$ action on $M$ is of type $\CC_3$ if and only if its 
 orbit space $M/\G$ is a geodesic 2-simplex with angles $\pi/2, \pi/3$ and $\pi/4$.
 
Our aim here is to take care of this exceptional case and prove

\begin{main}
Any polar $\G$ action on a simply connected positively curved manifold $M$ of type $\CC_3$ is equivariantly diffeomorphic to a polar action on a rank one symmetric space. This includes two actions on the Cayley plane where the universal covers of the associated chamber systems are not buildings.
\end{main}

Combining these results of course establishes, the

\begin{cor*}
 Any polar $\G$ action of cohomogeneity at least two on a simply connected closed positively curved manifold $M$ is equivariantly diffeomorphic to a polar $\G$ action on a rank one symmetric space.
\end{cor*}

This is in stark contrast to the case of cohomogeneity one, where in dimensions seven and thirteen there are infinitely many non-homogeneous manifolds (even up to homotopy). The classification work in \cite{GWZ} also lead to the discovery and construction of a new example of a positively curved manifold (see \cite{De} and \cite{GVZ}).

By necessity, as indicated above, the proof of Theorem B is entirely different from the proof of Theorem A. In general, the geometric realization of our chamber systems $\mathscr{C}(M;\G)$ utilized in the proof of Theorem A are not simplicial. However, in \cite{FGT} it was proved that in fact

\begin{main} \label{simplicial}
The geometric realization  $|\mathscr{C}(M,\G)|$ of a chamber system
$\mathscr{C}(M,\G)$ of type $\A_3$ or $\CC_3$ associated with a
simply connected polar $\G$-manifold $M$ is simplicial.
\end{main}

When the geometric realization of a chamber system of type $\M$ is simplicial it is called a \emph{Tits geometry} of type $\M$. This allows us to use an axiomatic characterization of $\CC_3$ geometries that are buildings 
 (see \cite{Ti2}, Proposition 9).
 So rather than considering the universal cover $\tilde{\mathscr{C}}(M;\G)$ directly, we construct in all but two cases a suitable cover of ${\mathscr{C}}(M;\G)$ (possibly ${\mathscr{C}}(M;\G)$ itself), and prove that it satisfies the $\CC_3$ building axiom of Tits. The two cases where this methods fails, are then recognized as being equivalent to two $\CC_3$ type polar actions on the Cayley plane  $\Bbb{OP}^2$ (cf. \cite{PTh, GK}).

We note, that since all our chamber systems $\mathscr{C}(M,\G)$ are homogeneous and those of type $\CC_3$ (and  $\A_3$) are Tits geometries an independent alternate proof of Theorem B follows from \cite{KL}.

\section{Preliminaries}\label{prelim}

The purpose of this section is threefold. While explaining the overall approaches to the strategies needed in the proof of Theorem B, we recall the basic concepts and establish notation.

\smallskip

Throughout $\G$ denotes a compact connected Lie group acting on a closed 1- connected positively curved manifold $M$ in a polar fashion and of type $\CC_3$. 

Fix a \emph{chamber} $C$ in a section $\Sigma$ for the action. Then $C$ is isometric to the orbit spaces $M/\G$ and $\Sigma/\W$, where $\W$ is the reflection group of $\Sigma$ and $\W$ acts simply transitively on the chambers of $\Sigma$. Since the action is of type $\CC_3$, $C$  is a convex positively curved 2-simplex with geodesic sides = faces,   $\ell_r, \ell_t$ and $\ell_q$ opposite its vertices $r, t$ and $q$ with angles $\pi/2, \pi/3$ and $\pi/4$ respectively. 

\smallskip

By the \emph{Reconstruction Theorem} of \cite{GZ} recall that any polar $\G$ manifold $M$ is completely determined by its so-called \emph{polar data}.
In our case, this data consist of $\G$ and all its isotropy groups,  \emph{together with their inclusions} along a chamber $C$ (cf. also Lemma 1.5 in \cite{Go}). We denote the principal isotropy group by $\H$, and the isotropy groups at vertices and opposite faces by  $\G_r, \G_t, \G_q$ and $\G_{\ell_r}, \G_{\ell_t}, \G_{\ell_q}$ respectively. What remains after removing
 $\G$ from this data will be referred to as the \emph{local data} for the action. 
 
 \smallskip
 
 With two exceptions, it turns out that only partial data are needed to show that the action indeed is equivalent to a polar action on a rank one symmetric space. Since the data in the two exceptional cases coincide with those of the exceptional $\CC_3$ actions on the Cayley plane, this will then complete the proof of Theorem A. In addition, it is worth noting, that since the groups $\G$ derived from those data (in \ref{lemC3d=2,3b} and \ref{(1,1,5)}) are maximal connected subgroups of $\F_4$, the identity component of the isometry group of the Cayley plane $\Bbb{OP}^2$, their actions are uniquely determined and turn out to be polar.
 
  \smallskip
 
 The proof of Theorem A in all but the two exceptional cases is based on showing that the universal cover, $\tilde{\mathscr{C}}$ of the \emph{chamber system} $\mathscr{C} = \mathscr{C}(M, \G)$  associated to the polar action is a spherical Tits building \cite{FGT}. Here, 
 the homogeneous chamber system  $\mathscr{C}(M, \G)$  is the union $\cup_{\g \in \G} \g C$ of all chambers with three \emph{adjacency relations} one for each face: Specifically $\g_1 C$ and $\g_2 C$ are $i$ adjacent if their respective $i$ faces are the same in $M$. This chamber system with the thin topology, i.e., induced from the its path metric is a simplicial complex by Theorem C, and hence $\mathscr{C}(M, \G)$ is a so-called $\CC_3$ geometry.

As indicated, the Fundamental Theorem of Tits used in \cite{FGT} to show that $\tilde{\mathscr{C}}$ is a building yields nothing for rank three chamber systems as well as rank three geometries. Instead we will show that $\mathscr{C}$, or a cover we construct of $\mathscr{C}$ is a $\CC_3$ building (and hence simply connected) by verifying an \emph{axiomatic incidence characterization} (see section \ref{axiom}) of such buildings due also to Tits.

\smallskip

The \emph{construction of chamber system covers} we utilize is equivalent in our context to the \emph{principal bundle construction} of \cite{GZ} (Theorem 4.5)  for Coxter polar actions and manifolds. Specifically for our case: 

\begin{itemize}
\item
Given the data, $\H, \G_{\ell_i}, \G_j,  i, j \in \{t,r,q\}$ and $\G$
for $(M,\G)$, the data for $(P,\L \times \G)$ consists of graphs  $\hat \H,  \hat
\G_{\ell _i},  \hat \G_j$ in $\hat\G:=\L \times \G$ of compatible homomorphisms
from $\H, \G_{\ell_i}, \G_j,  i, j \in {t,r,q}$ to $\L$. In particular, the local data for $(P,\L \times \G)$ are isomorphic to the local data for $(M, \G)$.

\item
Clearly $\L$ acts freely as a group of automorphisms, and
$\mathscr{C}(P,\L \times \G)/ \L = \mathscr{C}(M,\G)$, i.e.,  $
\widehat{\mathscr{C}}(M;\G):=  \mathscr{C}(P,\L \times \G)$ is a
chamber system covering of $\mathscr{C}(M;\G)$.
\end{itemize}

\no In our case  $\L$ will be $\S^1$ (or in one case $\S^3$).

\section{Basic tools and obstructions}

The aim of this section is to establish a number of properties and restrictions of the data to be used throughout. Unless otherwise stated $\G$ will be a compact connected Lie group and $M$ a closed simply connected positively curved manifold.

Without any curvature assumptions we have the possibly well known

\begin{lem}[Orbit equivalence]\label{oequiv}
Let $M$ be a simply connected polar $\G$ manifold. Then the slice representation of any isotropy group is orbit equivalent to that of its identity component.
\end{lem}

\begin{proof}
Recall that the slice representation of an isotropy group $\K = \G_p \subset \G$ restricted to the orthogonal complement $T_p^{\perp}$ of the fixed point set of $\K$ inside the normal space to the orbit $\G 
p$ is a polar representation. Clearly the finite group $\K/\K_0$ acts isometrically on the orbit space $\Sph(T_p^{\perp})/\K_0$, which is isometric to a chamber $C$ of the polar $\K_0$ action on the sphere $\Sph(T_p^{\perp})$. Since $C$ is convex with  non-empty boundary its soul point (the unique point at maximal distance to the boundary) is fixed by $\K/\K_0$. This soul point, however, corresponds to a principal $\K_0$ orbit, and hence to an exceptional $\K$ orbit unless $\K/\K_0$ acts trivially on $C$. However, by Theorem 1.5 \cite{AT} there are no exceptional orbits of a polar action on a simply connected manifold. 
\end{proof}

Because of this, when subsequently talking casually about a slice representation we refer to the slice representation of its identity component unless otherwise stated.

Using positive curvature the following basic fact was derived in \cite{FGT}, Theorem 3.2:

\begin{lem}[Primitivity]\label{prim}
The group $\G$ is generated by the (identity components) of the face isotropy groups of any fixed chamber.
\end{lem}

Naturally, the slice representations of $\G_t$, $\G_q$  and $\G_r$ play a fundamental role. We denote the respective kernels of these representations by $\K_t, \K_q$ and $\K_r$ and their quotients by $\bar\G_t$, $\bar\G_q$ and $\bar\G_r$. Since in particular the slice representation of $\G_t$ is of type $\A_2$ it follows that the multiplicity triple of the polar $\G$ manifold $M$,  i.e, the
dimensions of the unit spheres in the normal slices along the edges
$\ell_q, \ell_r, \ell_t$ is $(d, d, k)\in \Bbb Z_{+}^3$, where $d=1,
2, 4$ or $8$. 

For the kernels $\K_t$ and $\K_q$, which are usually
large groups, we have:

\begin{lem}[Slice Kernel] \label{effectiveC3}
Let $M$ be a simply connected polar $\G$-manifold of type $\CC_3$.
If $\G$ acts effectively,
 then the kernel $\K_t$, respectively $\K_q$ acts effectively on the
slices $T_q^\perp$ and $T_r^\perp$,  respectively $T_t^\perp$ and $T_r^\perp$.
\end{lem}

\begin{proof} Note that $\K_t$ fixes all sections through $t$
since $\K_t$ acts trivially on the slice $T^\perp _t$. We must prove that $\K_t\cap \K_q=\{1\}$, $\K_t\cap \K_r=\{1\}$ and $\K_q\cap \K_r=\{1\}$.
We consider only $\K_t\cap \K_q$,  since the arguments for the remaining cases are similar.

Note that since $\G$ is assumed to act effectively on $M$, and $\K_t\cap \K_q$ is contained in the
principal isotropy group,
 it suffices to prove that $\K_t\cap \K_q$ is normal in $\G$. By the primitivity  (see 
 \ref{prim}), $\G= \langle p_q^{-1}(\bar \G_{q,0}), p_t^{-1}(\bar \G_{t,0}) \rangle$, where $p_q:\G_{q} \to \bar
\G_{q}$ is the quotient homomorphism and $\bar \G_{q,0}$
is the identity component of $\bar\G_q$ and similarly for $p_t$. Thus, it suffices to show that $\K_t\cap \K_q$ is normal in each of 
$p_t^{-1}(\bar \G_{t,0})$ and $p_q^{-1}(\bar \G_{q,0})$. In each case, assuming the effective vertex isotropy group is connected does not alter the proof only simplifies notation. Accordingly, we proceed to assume that $\bar \G_t$ is connected, i.e., $\bar \G_t = \bar \G_{t,0}$ and will show that $\K_t\cap \K_q$ is a normal subgroup of $\G_t$.

Note that $\K_t\cap \K_q$ is a normal subgroup of $\K_t$ acting
trivially on both the slices $T_t^\perp$ and $T_q^\perp$.

By assumption  the 
quotient map $\G_{t,0}\subset \G_t\to \bar \G_t$ is surjective when
restricted to the identity component $\G_{t,0}$ of $\G_t$. A finite central cover
$\tilde \G_{t,0}$ of $\G_{t,0}$ is isomorphic to the product
$\tilde {\K}_{t,0}\times \tilde{{\bar \G}}_t$
where $\tilde {\K}_{t,0}$
is locally isomorphic to the identity component $\K_{t,0}$ of $\K_t$ and
$\tilde{{\bar \G}}_t$ is locally isomorphic to ${\bar \G}_t$.
In particular, $\G_t$
contains a connected and closed subgroup $\pi(\tilde {\bar{\G_t}})$ covering
$\bar \G_t$,
 where $\pi :\tilde \G_{t,0}\to \G_{t,0}$ is the cover
map.
Moreover, every element of the subgroup $\pi(\tilde {\bar{\G_t}})$ commutes with
the elements in $\K_{t,0}$.
On the other hand, for every $h\in \pi(\tilde {\bar{\G_t}})$,
the conjugation by $h$ gives rise to an element in the automorphism
group $\text{Aut}(\K_{t})$ since $\K_t$ is normal, hence defines a
homomorphism $\phi: \pi(\tilde {\bar{\G_t}})\to \text{Aut}(\K_{t})$.
Since $\phi(\pi (\tilde {\bar{\G_t}}))$ has a trivial image in
$\text{Aut}(\K_{t,0})$ under the forgetful homomorphism
$\text{Aut}(\K_{t})\to \text{Aut}(\K_{t,0})$, the group $\phi(\pi
(\tilde {\bar{\G_t}}))$ is  finite, and hence trivial
because  $\phi(\pi (\tilde {\bar{\G_t}}))$ is  connected. This
implies that the elements of $\pi(\tilde {\bar{\G_t}})$ commute with
the elements of $\K_{t}$. Since $\G_t=\langle \K_t, \pi (\tilde
{\bar{\G_t}})\rangle $ and $\K_t\cap \K_q$ is normal in $\K_t$, it then follows
 that $\K_t\cap \K_q$ is a
normal subgroup of $\G_{t}$. 

As mentioned above, the same arguments show that $\K_t\cap \K_q$ is normal in $p_t^{-1}(\bar \G_{t,0})$ in case ${\bar \G}_t$ is not connected. The same arguments also show that $\K_t\cap \K_q$ is normal in $p_q^{-1}(\bar \G_{q,0})$.
\end{proof}

\begin{rem}\label{t-rem}
It turns out that in all cases $\bar \G_t$ is connected. In fact, this is automatic whenever $d 
\ne 2$, since $\bar \G_t$ acts transitively on a projective plane. Up to local isomorphism its identity component is one of the
groups $\SO(3), \SU(3), \Sp(3)$, or $\F_4$ corresponding to $d = 1,2,4$ and $8$ respectively, and the slice
representation is its standard polar representation of type $\A_2$ (see also Table \ref{t-rep}). In view of the Transversality Lemma \ref{transv} below, $\G_t$ is connected whenever $k \ge 2$. In the $(2,2,1)$ case, the connectedness of $\G_r$ (again by Lemma \ref{transv}) implies that also in this case  $\bar{\G}_t$ is connected (see Proposition \ref{flip2}).
\end{rem}

The following simple topological consequence of transversality combined with the fact that the canonical deformation retraction of the orbit space triangle minus any side to its opposite vertex lifts to $M$ (or alternatively of the work \cite{Wie}) will also be used frequently:

\begin{lem}[Transversality]\label{transv}
Given a multiplicity triple $(d,d,m)$. Then the inclusion maps $\G/\G_r \subset M,  \G/\G_q \subset M$ and $\G/\G_{\ell_t} \subset M$ are $d$-connected, $\G/\G_{\ell_r} \subset M$, and $\G/\G_{\ell_r} \subset M$ are $\min\{d,m\}$ connected, and  $\G/\G_t \subset M$ is $m$-connected.
\end{lem}

Recall here that a continuos map is said to be \emph{$k$ - connected} if the induced map between the $i$th homotopy groups is an isomorphism for $i < k$ and a surjection for $i=k$.

Another Connectivity Theorem \cite{Wi3} (Theorem 2.1) using positive curvature \emph{\'{a} la Synge} is very powerful:

\begin{lem}[Wilking]\label{connect}
Let $M$ be a positively curved $n$-manifold and $N$ a totally geodesic closed codimension $k$ submanifold. Then the inclusion map $N \to M$ is $n-2k+1$ connected.

If in addition $N$ is fixed by an isometric action of a compact  Lie group $\K$ with principal orbit of dimension $m(\K)$, then the inclusion map is $n-2k+1+ m(\K)$ connected.
\end{lem}

We conclude this section with two severe restrictions on $\G$ stemming from positive curvature.

The first follow from the well known \emph{Synge type} fact, that an isometric $\T^k$ action has orbits with $\dim \le 1$ in odd dimensions  and $0$ in even dimensions, when $M$ has positive curvature (cf. \cite{Su}). In particular, since $\G_q$ has maximal rank among the isotropy groups, and the Euler
characteristic $\chi(\G/\G_q) > 0$ if and only if $\text{rk}(\G)= \text{rk}(\G_q)$ (\cite{HS} page 248) we conclude

\begin{lem}[Rank Lemma]\label{rank}
The dimension of $M$ is even if and only if $\text{rk}(\G)= \text{rk}(\G_q)$, and otherwise rank $\text{rk}(\G) = \text{rk}(\G_q)+1$.
\end{lem}

When adapting Wilking's \emph{Isotropy Representation} Lemma 3.1 from \cite{Wi2} for positively curved $\G$ manifolds to polar manifolds of type $\CC_3$ we obtain:

\begin{lem} [Sphere Transitive Subrepresentations] \label{WilkC3}
Let $\L_i\lhd \G_{\ell _i}$, $i\in \{q, r,t\}$ be a simple normal subgroup and $U$ an
irreducible isotropy subrepresentation of $\G/\L_i$. Then $(U, \L_i)$ is isomorphic to a
standard defining representation. In particular, $\L_i$ acts transitively on the sphere $\Sph (U)$.
\end{lem}

\begin{proof}
Let $\U$ be an irreducible isotropy subrepresentation of $\G/\L_i$
not isomorphic to a summand of the slice representation of $\L_i$ on
$T_i^\perp$. By \cite{Wi2}, $\U$ is isomorphic to a summand of the
isotropy representation of $\L_i^*/\L_i$, where $\L_i^*$ is a vertex
isotropy group. On the other hand, the almost effective factor of
$\L_i^*$ is well understood (cf. the tables \ref{t-rep} and \ref{q-rep}), which are all the
standard defining  representation. The desired result follows.
\end{proof}

\section{The $\CC_3$ building axiom}  \label{axiom} 

Recall that Tits has provided an axiomatic characterization of buildings of irreducible type $\M$ when  the geometric realization $|\mathscr{C}|$ ($\mathscr{C}$ with
the thin topology) of the associated chamber system  $\mathscr{C}$,
is a \emph{simplicial complex}. This characterization is given in
terms of the \emph{incidence geometry} associated with
$\mathscr{C}$.

The purpose of this section is to describe this characterization when $\M = \CC_3$ and translate it to our context.

 Here, by definition

\begin{itemize}
\item
Vertices $x,y \in |\mathscr{C}|$ are \emph{incident}, denoted $x*y$,
if and only if $x$ and $y$ are contained in a closed chamber of
$|\mathscr{C}|$.
\end{itemize}

Clearly, the incidence relation (not an equivalence relation) is
preserved by the action of $\G$ in our case.

To describe the needed characterization we will use the following standard terminology:

\begin{itemize}
\item
The {\it shadow} of a vertex $x$ on the set of vertices of type $i
\in I$, denoted $\text{Sh}_i(x)$, is the union of all vertices of
type $i$ incident to $x$.
\end{itemize}

Following Tits \cite{Ti2}, when $\M = \CC_3$, we call the vertices of type $q$, $r$ and
$ t$,  {\it points, lines}, and \emph{planes}  respectively. We denote
by $Q, R$ and $T$ the set of points, lines, and planes in
$\mathscr{C}(M;\G)$. Notice that $\G$ acts transitively on $Q$, $R$
and $T$. With this terminology the axiomatic characterization \cite{Ti2} (cf. Proposition 9 and the proof of the $\CC_3$ case on p. 544) alluded to above states:

\begin{thm}[$\CC_3$ Axiom]\label{c3axiom}
A connected  Tits geometry of type $\CC_3$  is a building if and
only if the following axiom holds:

\quad $\bullet$ \emph{(LL)}   If two lines are both incident to two
different points, they coincide.

\no Equivalently:

\quad $\bullet$  If $\emph{Sh}_Q(r) \cap \emph{Sh}_Q(r')$ has
cardinality at least two, then $r = r'$.

\no or:

\quad $\bullet$ For any $q, q' \in Q$, with $q \ne q'$,
$\emph{Sh}_R(q) \cap \emph{Sh}_R(q')$ has cardinality at most one.
\end{thm}

\medskip

 In our case, if $r \in R$ and $q \in Q$ are incident, (LL) is
clearly equivalent to

\medskip

$\bullet$ For any $r' \in \G_q(r), r'\ne r$, we have $\G_r(q) \cap
\G_{r'}(q) = q$

 \no or,

 $\bullet$ For any $q' \in \G_r(q), q'\ne q$, we have $\G_q(r) \cap \G_{q'}(r) = r$

\bigskip

We proceed to interpret (LL) in terms of the isotropy groups data.
This will be used either directly for $\mathscr{C}(M;\G)$ or
for a suitably constructed cover $\tilde{\mathscr{C}}(M;\G)$ as
described at the end of section \ref{prelim}. For notational simplicity we will describe it here only for
$\mathscr{C}(M;\G)$ (for the general case see remark \ref{lift crit} below).

\begin{prop} \label{property}
If $\mathscr{C}(M;\G)$ is a building of type $\CC_3$, then the
following holds:

$\star$  for any pair of different points $q$, $q'\in Q$ both
incident to an $r\in R$, we have
$$\G_{q}\cap \G_{q'}\subset \G_{rq}\cap \G_{rq'}$$
where $\G_{rq}$ denotes the isotropy group of the unique edge
between $r$ and $q$ \emph{(cf. Theorem C)}.
\end{prop}

\begin{proof} Note that
every line in the orbit $\G_{q}\cap \G_{q'}(r)$ is incident to both
$q$ and $q'$. Axiom (LL) implies that the orbit contains only one
line, $r$ and hence $\G_{q}\cap \G_{q'}\subset \G_r$. Since
$\mathscr{C}(M;\G)$ is a building,  we have $\G_r\cap \G_q=\G_{rq}$
and $\G_r\cap \G_{q'}=\G_{rq'}$. The desired result follows.
\end{proof}

We will see that the condition $\star$ together with an assumption
on a suitable \emph{reduction} of the $\G$ action implies that
$\mathscr{C}(M;\G)$ is a building of type $\CC_3$.

To describe the reduction, let $r\in R$ be a line, and let $\Sph
^\perp _{r, Q}$ be the normal sphere in the summand  in the slice
$T^\perp_r$. Then the shadow of $r$ in $Q$ is $\exp (\frac{\pi}{4}
\Sph ^\perp _{r, Q})$.  Moreover, the isotropy group $\G_r$ acts
transitively on $\Sph ^\perp _{r, Q}$.
\begin{center}Let $\K_{r,Q}$ denote the identity component of the kernel of the
transitive $\G_r$ action on $\Sph ^\perp _{r, Q}$.\end{center}
 It is clear that
the fixed point connected component $M^{\K_{r,Q}}$ (containing $r$)
is a cohomogeneity one $\N_0(\K_{r,Q})$ submanifold of $M$, where 
$\N_0(\K_{r,Q})$ is the
identity component of the normalizer $\N(\K_{r,Q})$ of $\K_{r,Q}$ in $\G$. The
corresponding chamber system denoted $\mathscr{C}(M^{\K_{r,Q}})$ is
a subcomplex of $\mathscr{C}(M) := \mathscr{C}(M;\G)$ that inherits
an incidence structure, which gives rise to  a Tits geometry of rank
$2$.

\begin{lem}[Reduction] \label{Criterion1}
The connected chamber system $\mathscr{C}(M,\G)$ of type $\CC_3$ is
a building if for any $r \in R$, the reduction
$\mathscr{C}(M^{\K_{r,Q}})$ is a $\CC_2$-building and $\star$ holds.
\end{lem}

\begin{proof}
If not, by Axiom (LL) there are two points $q\ne q'\in Q$ which are
both incident to two different lines $r, r'\in R$. By $\star$ we
know that $\G_{q}\cap \G_{q'}\subset \G_{rq}\cap \G_{rq'}$ and
$\G_{q}\cap \G_{q'}\subset \G_{r'q}\cap \G_{r'q'}$. Therefore, the
configuration $\{rq, rq', r'q, r'q'\}$ is contained in the fixed
point set $M^{\G_q\cap\G_{q'}}$. Since by definition clearly
$\K_{r,Q} $ is a subgroup of $\G_{q}\cap \G_{q'}$, we have that
$M^{\G_q\cap\G_{q'}} \subset M^{\K_{r,Q}}$. This implies that there
is a length $4$ circuit in the $\CC_2$ building
$\mathscr{C}(M^{\K_{r,Q}})$. A contradiction.
\end{proof}

The following technical criterion will be more useful to us:

\begin{lem}[$\CC_3$ Building Criterion] \label{Criterion2}
The connected chamber system $\mathscr{C}(M,\G)$ is
a building if for any $r \in R$, the reduction
$\mathscr{C}(M^{\K_{r,Q}})$ is a $\CC_2$-building and the following
\emph{Property (P)} holds:

\emph{(P)} For any $q\in \emph{Sh}_Q(r)$, and any Lie group $\L$ with
$\K_{r,Q}\subset \L\subset \G_q$ but $\L \nsubset \G_{rq}$, the
normalizer $\N(\K_{r,Q})\cap \L$ is not contained in $\G_{rq}$
either.
\end{lem}

\begin{proof}
By the previous lemma it suffices to verify $\star$. Suppose $\star$
is not true. Then there is an $r\in R$ and a pair of points $q\ne
q'$ both incident to $r$ such that $\G_q\cap \G_{q'}$ is not a
subgroup of $\G_{rq}$. Let $\L=\G_q\cap \G_{q'}$. By Assumption (P),
there is an $\alpha \in \N(\K_{r,Q})\cap \L$ so that $\alpha\notin
\G_{rq}$. However, $\G_r\cap \G_q\cap  \N(\K_{r,Q})=\G_{rq} \cap
\N(\K_{r,Q})$ since $M^{\K_{r,Q}}$ is an $\CC_2$ building. In
particular, $\alpha \notin \G_r$, and so there is a length $4$
circuit  $\{rq, q\alpha(r), \alpha(r)q', q'r\}$ in the $\CC_2$
building $\mathscr{C}(M^{\K_{r,Q}})$. A contradiction.
\end{proof}

\begin{rem}\label{lift crit}
For an $\S^1$ cover $\tilde {\mathscr{C}} : =
\mathscr{C}(P,\S^1\times \G)$ of $\mathscr{C}(M,\G)$ constructed as
above note that the property $\star$ is inherited from $(M, \G)$.
Likewise, the group $\hat \K$ being the graph of the homomorphism
$\G_{\ell_t} \subset \G_r$ to $\S^1$ restricted to $\K: = \K_{r,Q}$
satisfies Property (P) when $\K$ does.  For this note that by
construction the local data for the reduction $P^{\hat \K}$ are
isomorphic to the local data for $M^{\K}$. It then follows as in the
proofs above, that if a component of the reduction
$\mathscr{C}(P^{\hat \K}) \subset \mathscr{C}(P,\S^1\times \G)$ is a
$\CC_2$-building, then the corresponding component of
 $\tilde{\mathscr{C}}$ will be a $\CC_3$ building covering
$\mathscr{C}(M,\G)$, and our main result, Theorem 4.10, from the \cite{FGT} applies.
\end{rem}

\begin{rem}\label{rank3crit}
If $\K' = \K' _{r,Q}\subset \K_{r, Q}=\K$ is a subgroup, then the
assumption of $\mathscr{C}(M^{\K})$ being a $\CC_2$ building in the
above criterion may be replaced by, the fixed point component
${\mathscr{C}}(M^{\K'})\supset {\mathscr{C}}(M^{\K})$ being a
$\CC_2$ building, or a rank $3$ building. For the latter,  we notice
that, by Charney-Lytchak [CL] Theorem 2, a rank $3$ spherical building is a
CAT(1) space, hence any two points of distance less than $\pi$ are
joined by a unique geodesic. This clearly excludes a length $4$
circuit in the above proof, since its perimeter is $\pi$.
\end{rem}

\begin{rem}\label{edge ker}
Note that clearly $\K_t  \subset \K_{\ell_q} \cap \K_{\ell_r}$ and similarly for the other kernels of vertex and edge isotropy groups. In particular, for the identity component $\K'$ of $\K_t$ we have $\K' \subset \K$, where $\K$ ($=\K_{r,Q}$) is the identity component of the kernel of $\G_{r}$ acting on $\Sph^d$. Consequently, the reduction $M^{\K'}$ is a cohomogeneity two manifold of type either $\A_3$, or $\CC_3$ containing the cohomogeneity one manifold $M^{\K}$ (cf. \ref{rank3crit} above).
\end{rem}

\section{Classification outline and organization}

The subsequent sections are devoted to a proof of the following main result of the paper:

\begin{thm} \label{thmC3}
Let $M$ be a compact, simply connected positively curved  polar
$\G$-manifold with associated chamber system $\mathscr{C}(M;\G)$ of
type $\CC_3$. Then the universal cover $\tilde{\mathscr{C}}$  of
$\mathscr{C}(M;\G)$ is a building if and only if $(M,\G)$ is not
equivariantly diffeomorphic to one of the exceptional polar actions on
$\Bbb O\Bbb {P}^2$ by $\G = \SU(3)\cdot \SU(3)$ or $\G = \SO(3)\cdot \G_2$.\end{thm}

This combined with the main result of \cite{FGT} proves Theorem B in the introduction.

The purpose of this section is to describe how the proof is organized according to four types of scenarios driven by the possible compatible types of slice representations for $\G_t$ and $\G_q$ at the vertices $t$ and $q$ of a chamber $C$.

The common feature in each scenario and all cases is the determination of all local data. The basic input for this is indeed knowledge of the slice representations
at the vertices $t$ and $q$ of a chamber $C$, and Lemma
\ref{effectiveC3}. The local data identifies the desired $\K \subset
\G_r$ reduction $M^{\K}$ with its cohomogeneity one action by
$\N(\K)$ referred to in the Building Criteria Lemma
\ref{Criterion2}, with Property (P) being essentially automatic.
 The main difficulty is to establish that $\mathscr{C}(M^{\K}) \subset \mathscr{C}(M;\G)$ or the corresponding reduction in a cover
 (which by construction has the same local data) is a $\CC_2$ building. The first step for this frequently uses the following consequence of the
 classification work on positively curved cohomogeneity one manifolds in \cite{GWZ} and \cite{Ve}.

\begin{lem}\label{GWZ} Any simply connected positively curved cohomogeneity one manifold with multiplicity pair different from $(1, 1)$, $(1, 3)$ and $(1, 7)$ is equivariantly diffeomorphic to a rank one symmetric space.
\end{lem}

As already pointed out and used,   there are only four
possible (effective) slice representations at $t$, in particular forcing the
codimensions of the orbit strata corresponding to $\ell_q, \ell_r,$
and  $\ell_t$ to be $d+1, d+1$ and $k+1$, where $d=1, 2, 4$ or $8$.
In Table \ref{t-rep}, $\L^{\pm}$, respectively $\H$ are the singular, respectively principal isotropy groups for the effective slice representation, $\chi$ by $\bar \G_t$ restricted to the unit sphere, and $l_{\pm}+1$ are the codimensions of the singular orbits.

{\setlength{\tabcolsep}{0.10cm}
\renewcommand{\arraystretch}{1.6}
\stepcounter{equation}
\begin{table}[!h]
      \begin{center}
       \begin{tabular}{|c||c|c|c|c|c|c|c|}
\hline
    $n$     &$\bar \G_t$ &$\chi$   &$\L^{-}$    &$\L^{+}$     &$\H$& $(l_-,l_+)$ & $W$   \\
\hline \hline

$4$ & $\SO(3)$ &  &$\S(\O(2)\O(1))$ & $\S(\O(1)\O(2))$ &
$\Z_2\oplus\Z_2$ & $( 1,1 )$ & $ \A_2 $\\
\hline

$7$ & $\PSU(3)$ & $\text{Ad}$ &
$\S(\U(2)\U(1))/\Delta(\Z_3)$ & $\S(\U(1)\U(2))/\Delta(\Z_3)$ & $\T^2/\Z_3$& $( 2,2 )$ & $\A_2$ \\
\hline

$13$ & $\Sp(3)/\Delta (\Z_2)$ & $\psi_{14}$ &
$\Sp(2)\Sp(1)/\Delta(\Z_2)$ & $\Sp(1)\Sp(2)/\Delta(\Z_2)$ & $\Sp(1)^3/\Delta(\Z_2)$& $( 4,4 )$ & $\A_2$ \\
\hline

$25$ & $\F_4$ & $\psi_{26}$ &
$\Spin(9)$ & $\Spin(9)$ & $\Spin(8)$& $(8 ,8 )$ & $\A_2$ \\

 \hline

          \end{tabular}
      \end{center}
      \vspace{0.1cm}
      \caption{Effective $t$-slice representations on $\Sph^{\perp}_t = \Sph^n$}\label{t-rep}
\end{table}}

Similarly  (see Table \ref{q-rep}), the identity component $ (\bar \G_q)_0$ of possible effective $\CC_2$ type slice representations
at $q$ which are compatible with the multiplicity restrictions in Table \ref{t-rep} are
known as well (see e.g. Table E of \cite{GWZ} in which we have
corrected an error for the exceptional $\SO(2)\Spin(7)$ representation (see also \cite{GKK} (Main Theorem)).

{\setlength{\tabcolsep}{0.04cm}
\renewcommand{\arraystretch}{1.4}
\stepcounter{equation}
\begin{table}[!h]
      \begin{center}
          \begin{tabular}{|c||c|c|c|c|c|c|c|}
\hline
    $n$     &$ (\bar \G_q)_0$ &$\chi$   &$\L^{-}$    &$\L^{+}$     &$\H$& $(l_-,l_+)$ & $W$   \\
\hline \hline

$8k+15$, $k\ge 0$ & $\frac{\Sp(2)\Sp(k+2)}{\Delta(\Bbb Z_2)}$ &
$\nu_2\hat{\otimes}\nu_{k+2}$ & $\frac{\Sp(2)\Sp(k)}{\Delta(\Bbb
Z_2)}$ & $\frac{\Sp(1)^2\Sp(k+1)}{\Delta(\Bbb Z_2)}$ &
$\frac{\Sp(1)^2\Sp(k)}{\Delta(\Bbb Z_2)}$ & $( 4,4k+3 )$ & $\CC_2$
\\

\hline
\hline

$4k+7$ , $k\ge 1$ even & $\frac{\SU(2)\SU(k+2)}{\Delta(\Bbb Z_2)} $
& $\mu_2\hat{\otimes}\mu_{k+2}$ &
$\frac{\triangle\SU(2)\SU(k)}{\Delta(\Bbb Z_2)}$ &
${\S^1\cdot\SU(k+1)}$ & $\frac{\S^1\cdot\SU(k)}{\Delta(\Bbb Z_2)}$ &
$(2 ,2k+1 )$ & $\CC_2$
\\

\hline

$4k+7$ , $k\ge 1$ odd & ${\SU(2)\SU(k+2)} $ &
$\mu_2\hat{\otimes}\mu_{k+2}$ & $ {\triangle\SU(2)\SU(k)} $ & $
{\S^1\cdot\SU(k+1)} $ & $ {\S^1\cdot\SU(k)} $ & $(2 ,2k+1 )$ &
$\CC_2$
\\

\hline

$4k+7$, $k\ge 1$ & $\frac{\U(2)\SU(k+2)}{\Delta(\Bbb Z_k)} $ &
$\mu_2\hat{\otimes}\mu_{k+2}$ &
$\frac{\triangle\U(2)\SU(k)}{\Delta(\Bbb Z_k)}$ &
$\frac{\T^2\cdot\SU(k+1)}{\Delta(\Bbb Z_k)}$ &
$\frac{\T^2\cdot\SU(k)}{\Delta(\Bbb Z_k)}$ & $(2 ,2k+1 )$ & $\CC_2$
\\

\hline

$7$ & $\frac{\U(2)\SU(2)}{\Delta(\Bbb Z_2)}$ &
$\mu_2\hat{\otimes}\mu_2$ &
$\triangle\SO(3)$ & $\T^2$ & $\S^1$ & $(2,1)$ &  $\CC_2$ \\

\hline
\hline

$2k+3$, $k\ge 1$ even & $\frac{\SO(2)\SO(k+2)}{\Delta (\Bbb Z_2)}$ &
$\rho_2\hat{\otimes}\rho_{k+2}$ &
$\frac{\triangle\SO(2)\SO(k)}{\Delta (\Bbb Z_2)}$ & $\SO(k+1)$ &
$\SO(k)$
& $( 1,k )$ & $\CC_2$\\

\hline

$2k+3$, $k\ge 1$ odd & $\SO(2)\SO(k+2)$ &
$\rho_2\hat{\otimes}\rho_{k+2}$ & $\triangle\SO(2)\SO(k)$ &
$\Z_2\cdot\SO(k+1)$ & $\Z_2\cdot\SO(k)$
& $( 1,k )$ & $\CC_2$\\

\hline
\hline

$13$ & $\SO(2)\G_2$ & $\rho_2\hat{\otimes}\phi_7$ &
$\triangle\SO(2)\SU(2)$ & $\Z_2\cdot\SU(3)$ & $\Z_2\cdot\SU(2)$ & $(
1,5 )$ & $\CC_2$\\
\hline

$15$ & $\frac{\SO(2)\Spin(7)}{\Delta(\Bbb Z_2)}$ &
$\rho_2\hat{\otimes}\Delta_7$ & $\triangle\SO(2)\SU(3)$ & $\G_2$ &
$\SU(3)$ &
$(1 ,6 )$ & $\CC_2$ \\
\hline

$9$ & $\SO(5)$ & $ad$ &
$\U(2)$ & $\SO(3)\SO(2)$ & $\T^2$& $(2 ,2 )$ & $\CC_2$ \\

\hline

$19$ & $\SU(5)$ & $\Lambda^2\mu_5$ &
$\Sp(2)$ & $\SU(2)\SU(3)$ & $\SU(2)^2$ & $( 4, 5 )$ & $\CC_2$\\

& $\U(5)$ & $\Lambda^2\mu_5$ &
$\S^1\cdot\Sp(2)$ & $\S^1\cdot\SU(2)\SU(3)$ & $\S^1\cdot\SU(2)^2$ &  &  \\
\hline

          \end{tabular}
      \end{center}
      \vspace{0.1cm}
      \caption{Effective $q$-slice representation  on
      $\Sph^{\perp}_q = \Sph^{n}$}\label{q-rep}
\end{table} }

Aside from a few
exceptional representations, they are the isotropy representations
of the Grassmannians $\G_{2,m+2}(k)$ of $2$-planes  in
$k^{m+2}$, where $k=\Bbb R$, $\C$, or  $\Bbb H$. The pairs of
multiplicities that occur for the exceptional representations are
$(1, 6), (1, 5), (4, 5), (2, 2)$, corresponding to
$\bar\G_q=\SO(2)\Spin(7), \SO(2)\G_2$, $\SU(5), \U(5)$, or $\SO(5)$.

\bigskip

Note that effectively, there are only four \emph{exceptional} $\G_q$ slice representations, corresponding to the last four rows of Table \ref{q-rep}. However, \emph{special situations} occur also when the slice representation of $\bar \G_q$ is the isotropy representation of the real Grassmann manifold, when its multiplicity $(1,k)$ happen to have $k = d = 1,2,4$ or $8$. We will refer to these as \emph{ flips}. As may be expected, the low multiplicity cases $(1,1,1)$, $(1,1, 5)$ and $(2,2,3)$ play important special roles. The latter two are where the exceptional Cayley plane emerges, the only cases where complete information about the polar data are required.

Accordingly we have organized the proof of \ref{thmC3} into four sections depending on the type of slice representations we have along $Q$: Three Grassmann flips, three Grassmann series (two non minimal), two minimal Grassmann representations, and four exceptional representations.

\section{Grassmann Flip $\G_q$ slice representation}\label{flip}

This section will deal with the multiplicity cases $(d,d,1)$ with $d = 2,4$ and $8$, leaving $d=1$ (minimal and odd) for section \ref{minimal}. We have the following common features:

\begin{lem}\label{flip iso}
The isotropy groups $\G_q$ and $\G_r$ are connected, and  the reducible $\bar \G_r $ slice representation on $\Sph^1 * \Sph^d = \Sph^{d+2}$ is the standard action by $\SO(2)\SO(d+1)$. For the kernels of the slice representations we have that $\K_t = \{1\}$,  $\K_q = \K_{\ell_r}$ and $\K_r = \K_{\ell_q}$.
\end{lem}

\begin{proof}
 The Transversality Lemma \ref{transv} implies that the orbits  $Q = \G q$ and $R = \G r$ are simply connected since $M$ is. In particular, $\G_q$ and $\G_r$ are connected since $\G$ is. The second claim follows since $d$ is even (cf. Appendix in \cite{FGT} for a description of reducible polar representations).

 Since $(\bar \G_q)_{\ell_r}$ (cf. Table \ref{q-rep}) as well as $(\bar \G_r)_{\ell_q}$ act effectively on the respective normal spheres $\Sph^d$, we see that $\K_q = \K_{\ell_r}$ and $\K_r = \K_{\ell_q}$. Also since $\K_t \subset \K_{\ell_r}$ we have $\K_t \cap \K_q = \K_t$ but $\K_t \cap \K_q = \{1\}$  by the Kernel Lemma  \ref{effectiveC3} and hence $\K_t = \{1\}$.
\end{proof}

Recall that $\K$ is the identity component of the kernel of the $\G_r$ action restricted to $\S^d$.

\begin{lem}\label{flip red}
Clearly $\K \lhd \G_r$, and $\K \subset \G_{\ell_t}$ acts transitively on the corresponding normal sphere $\Sph^1$ with kernel identity component of $\K_r$. Moreover, $\K \cap \K_q = \{1\}$ and hence  $\K \subset \G_q \to \bar \G_q$ is injective.

The reduction $M^{\K}$ is a positively curved irreducible cohomogeneity one $\N_0(\K)$ manifolds with multiplicity pair $(d,1)$.
\end{lem}

\begin{proof}
Note that $\K \cap \K_q$ acts trivially on $\Sph^1 * \Sph^d$, so $\K \cap \K_q \subset \K_r$. The second claim follows since $\K_r \cap \K_q = \{1\}$.

Since $\K \lhd \G_{\ell_t} \to \bar \G_q$ is injective, we see from Table
 \ref{q-rep} that $\N(\K) \cap \G_q/ \G_{\ell_t} =  \N(\K) \cap \bar \G_q/ \bar \G_{\ell_t} = \Sph^1$, and hence $M^{\K}$ is cohomogeneity one with multiplicity pair $(d,1)$.

To complete the proof assume by contradiction that the action is reducible, i.e., that the action by $\N_0(\K)/\K$ on $M^{\K}$ is equivalent to the sum action of $\SO(2)\SO(d+1)$ on $\Sph^1 * \Sph^d$,
where the isotropy $(\N_0(\K)/\K)_{q}$ is $\SO(2)\SO(d)$.  In
all cases, it is easy to see that, the center of $\G_q$ intersects the center of $\N_0(\K)$ in a nontrivial subgroup $\S^1$.
This, together with primitivity implies that, $\S^1$ is in the
center of $\G$.  Notice that, as a subgroup of $\G_q$, $\S^1$ can
not be in $\K_q$ because $\K_q\subset \H$, and the factor $\SO(2)\lhd \bar \G_q$ acts freely on the unit sphere of the
slice $T_q^\perp$. Thus, the fixed point set $M^{\S^1}$ coincides with the orbit $\G q=\G/\G_q$.  From the
classification of positively curved  homogeneous spaces we  get
immediately that, $\G $ is the product of $\S^1$ (or $\T^2$ if $d=2$)
with one of a few orthogonal groups or unitary groups,  each of which is
not big enough to contain the simple group $\G_t$. The desired
result follows.
\end{proof}
\medskip

Although what remains is in spirit the same for all the flip cases, we will cary out the arguments for each case individually, beginning with $d= 8$.

\begin{prop}\label{flip8}
In the Flip \emph{(8,8,1)} case,  $\mathscr{C}(M,\G)$ or an $\S^1$
covering is a building, with the isotropy representation
of $\E_7/\E_6\times \S^1$ as a  linear model.
\end{prop}

\begin{proof}
From Lemma \ref{flip iso} and Tables \ref{q-rep} and \ref{t-rep} we obtain the following information about the local data:  $\G_t=\F_4 \supset \Spin(8) = \H$, $\G_{\ell_q} = \Spin(9)$, $\G_{\ell_r} = \Spin(9)$, $\G_q =\S^1\cdot \Spin(10) \supset \Delta (\S^1)\cdot \Spin(8) = \G_{\ell_t}$, and $\G_r= \S^1 \cdot \Spin(9)$.

Also $\G_r \rhd \K= \Delta (\S^1) \subset \G_{\ell_t} \subset \G_q$, and from Lemma \ref{flip red} and Lemma \ref{GWZ} we see that the corresponding reduction, $M^\K$  is $\Sph^{19}$, $\Sph^{19}/\Z_m$, or $\Sph^{19}/\S^1 = \CP^9$ with the tensor product representation by $\SO(2)\SO(10)$ of type $\CC_2$ or induced by it. It is easily
seen that the Assumption (P) in Lemma \ref{Criterion2} is
satisfied as well. In particular, if $M^{\K}=\Sph^{19}$, the associated chamber system $\mathscr{C}(M^{\K})$ is the a building of type $\CC_2$ and by Lemma \ref{Criterion2}  we conclude that $\mathscr{C}(M,\G)$ is a building.

For the latter two cases, we will use the bundle construction for polar actions to obtain a free $\S^1$ covering of $\mathscr{C}(M;\G)$. Guided by our knowledge of the cohomogeneity one diagrams , i.e., data for the cohomogeneity one manifolds $\Sph^{19}/\Bbb Z_m$ or $\Bbb {CP}^9$ we proceed as follows:

Note that since $\G_t$, $\G_{\ell _r}$ and $\G_{\ell _q}$ are simple groups, only the trivial homomorphism to $\S^1$ exists. Now let
$\hat \G_q, \hat \G_r$ be the graphs of the projection
homomorphisms $\G_q\to \S^1$, and $\G_r\to  \S^1$. We denote the total space of 
the
corresponding principal $\S^1$ bundle over $M$ by $P$. Then $P$ is a polar $\S^1\cdot \G$ manifold, and $\mathscr{C}(P;\S^1\cdot \G)$ covers  $\mathscr{C}(M;\G)$.

Let $\hat \K \subset \hat \G_{\ell _t}$ be the graph of $\K$ in $\S^1\cdot\G$. From \ref{lift crit} and our choice of data in $\S^1\cdot\G$ it follows that $P^{\hat \K}\to
M^\K$ is the Hopf bundle if $M^\K=\Bbb {CP}^{9}$, and the bundle
$\S^1\times _{\Bbb Z_m}\Sph^{19} \to \Sph^{19}/\Bbb Z_m$ if
$M^\K=\Sph^{19}/\Bbb Z_m$. In the former case,  $\mathscr{C}(P^{\hat \K})$ is the $\CC_2$ building $\mathscr{C}(\Sph^{19}, \SO(2)\SO(10))$ and we are done by Lemma \ref{Criterion2} via \ref{lift crit}. In the latter case, the action on the reduction $P^{\hat \K}$ is not primitive, so $\mathscr{C}(P^{\hat \K})$ is not connected. However, each connected component is the $\CC_2$ building $\mathscr{C}(\Sph^{19}, \SO(2)\SO(10))$ and hence by \ref{lift crit} the corresponding component of $\mathscr{C}(P)$ is a $\CC_3$ building covering $\mathscr{C}(M)$. When combined with the previous section, this in turn shows that $M^\K$ cannot be a lens space when $M$ is simply connected.
\end{proof}

\begin{prop}\label{flip4}
In the Flip \emph{(4,4,1)} case  $\mathscr{C}(M,\G)$ or an $\S^1$
covering is a building, with the isotropy representation of $\SO(12)/\U(6)$ as a linear model.
\end{prop}

\begin{proof}
From Lemma \ref{flip iso} and Tables \ref{q-rep} and \ref{t-rep} we obtain the following information about the local data modulo a common $\Z_2$ kernel:  $\G_t= \Sp(3) \supset \Sp(1)^3 = \H$, $\G_{\ell_q} =  \Sp(1)\Sp(2)$, $\G_{\ell_r} =  \Sp(2)\Sp(1)$, $\G_q = \S^1\Spin(6)\Sp(1) \supset \Delta (\S^1)\cdot \Spin(4)\Sp(1) = \G_{\ell_t}$, and $\G_r= \S^1\Spin(5)\Sp(1)$.

In this case $\G_r \rhd \K= \Delta (\S^1)\Sp(2) \lhd \G_{\ell_t} \subset \G_q$, and from Lemma \ref{flip red} and Lemma \ref{GWZ} we see that the corresponding reduction, $M^\K$  is $\Sph^{11}$, $\Sph^{11}/\Z_m$, or $\Sph^{11}/\S^1 = \CP^5$ with the linear tensor product representation by $\SO(2)\SO(6)$ of type $\CC_2$ or induced by it. It is easily
seen that the Assumption  (P) in Lemma \ref{Criterion2} is
satisfied as well. In particular, if $M^{\K}=\Sph^{11}$, the associated chamber system $\mathscr{C}(M^{\K})$ is the a building of type $\CC_2$ and by Lemma \ref{Criterion2}  we conclude that $\mathscr{C}(M,\G)$ is a building.

If $M^{\K}=\Bbb {CP}^5$ or a lens space $\Sph^{11}/\Bbb Z_m$, we proceed as above with an $\S^1$ bundle construction.  Again only  the trivial homomorphism to $\S^1$ exists from $\G_t$, $\G_{\ell _r}$ and $\G_{\ell _q}$, and we choose $\hat \G_q, \hat \G_r$ to be the graphs of the projection
homomorphisms $\G_q\to \S^1$, and $\G_r\to  \S^1$. We denote the total space of 
the
corresponding principal $\S^1$ bundle over $M$ by $P$. As above, $P$ is a polar $\S^1\cdot \G$ manifold, and $\mathscr{C}(P;\S^1\cdot \G)$ covers  $\mathscr{C}(M;\G)$.

 From \ref{lift crit} and our choice of data in $\S^1\cdot\G$ it follows that $P^{\hat \K}\to
M^\K$ is the Hopf bundle if $M^\K=\Bbb {CP}^{5}$, and the bundle
$\S^1\times _{\Bbb Z_m}\Sph^{11} \to \Sph^{11}/\Bbb Z_m$ if
$M^\K=\Sph^{11}/\Bbb Z_m$. The proof is completed as above.
\end{proof}

\begin{prop}\label{flip2}
In the Flip \emph{(2,2,1)} case  $\mathscr{C}(M,\G)$ or an $\S^1$
covering is a building, with the isotropy representation of $\SU(6)/\S(\U(3)\U(3))$ as a linear model .
\end{prop}

\begin{proof}
We begin by verifying our earlier claim (see \ref{t-rem}) that $\bar \G_t$ is connected also in this case. From (\ref{flip 
iso}) we already know that $\G_r$ and hence $\bar \G_r$ is connected, and that its slice representation is
 the product action of 
$\bar \G_r= \SO(3)\times \SO(2)$ on $\Bbb R^3\oplus \Bbb R^2$. The singular isotropy group along $\Bbb R^2$ (away from origin) is $\SO(3)$. Hence, the isotropy group $\bar \G_{\ell _q} = \SO(3)$.  

On the other hand, suppose $\bar \G_t $ is not connected. Then, by \ref{flip iso}  $\bar \G_t = \G_t = \PSU(3)\rtimes \Bbb Z_2$ and $\G_{\ell_q} = (\S(\U(2)\U(1))/\Z_3) \rtimes \Bbb Z_2$. In particular the slice representation along $\ell_q$ is by $\bar \G_{\ell_q} = \PSU(2)\rtimes \Bbb Z_2$ acting on $\Sph^2 = \CP^1$ where $\Bbb Z_2$ acts by complex conjugation. Contradicting $\bar \G_{\ell_q} = \SO(3)$.

The above and Tables \ref{q-rep} and \ref{t-rep} yield the following information about the local data modulo the $\Z_3$ kernel:  $\G_t=\SU(3) \supset \T^2 = \H$, $\G_{\ell_q} = \S(\U(2)\U(1)) = \U(2)$, $\G_{\ell_r} = \S(\U(1)\U(2)) = \U(2)$, $\G_q=\U(2)\U(2) $.  Moreover, $\G_{\ell_t} = \T^3$
  and $\G_r=\S^1\cdot \U(2)$,
where the $\U(2)$ factor in $\G_r$ is the face isotropy group of
$\G_{\ell _q}$. 

Here, $\G_r \rhd \K= \T^2 \lhd \G_{\ell_t} \subset \G_q$, and from Lemma \ref{flip red} and Lemma \ref{GWZ} we see that the corresponding reduction, $M^\K$  is $\Sph^{7}$, $\Sph^{7}/\Z_m$, or $\Sph^{7}/\S^1 = \CP^3$ with the linear tensor product representation by $\SO(2)\SO(4)$ of type $\CC_2$ or induced by it. Again, the Assumption  (P) in Lemma \ref{Criterion2} is easily checked to hold. In particular, if $M^{\K}=\Sph^{7}$, we conclude as above that $\mathscr{C}(M,\G)$ is a building.

For the latter two cases, we are again guided by the reduction for our bundle construction. For $\hat \G_t $ we have no choice but \ $\hat \G_t=\{1\}\cdot \G_t$. We let $\hat \G_q$ be the graph of
the homomorphism $\U(2)\U(2)\to \S^1$ defined by sending $(A, B)$ to
$\det(A)\det(B)^{-1}$, and  $\hat \G_r $ the graph
of the projection homomorphism $\G_r= \S^1\cdot \U(2) \to \S^1$. This yields a compatible choice of data for a polar
$\S^1\cdot\G$ action on a principal $\S^1$ bundle $P$ over $M$ whose corresponding  chamber system
$\mathscr{C}(P;\S^1\cdot\G)$ is a free $\S^1$ cover of  $\mathscr{C}(M,\G)$.

 Again from \ref{lift crit} and our choice of data in $\S^1\cdot\G$ it follows that $P^{\hat \K}\to
M^\K$ is the Hopf bundle if $M^\K=\Bbb {CP}^{3}$, and the bundle
$\S^1\times _{\Bbb Z_m}\Sph^{7} \to \Sph^{7}/\Bbb Z_m$ if
$M^\K=\Sph^{7}/\Bbb Z_m$, and the proof is completed as above.
\end{proof}

\begin{rem}
The tensor representation of $\SU(3)\SU(3)$ on $\Bbb
C^3\otimes \C^3$ is not polar, but it is polar on the projective
space $\Bbb P(\Bbb C^3\otimes \C^3)$. On the other hand, it is necessary in the above
construction of the covering that both $\G_q$ and
$\G_r$ have $\T^2$ factors, since the face isotropy groups $\G_{\ell
_r}\cong \G_{\ell _q}\cong \U(2)$ which are subgroups in
$\G_t=\SU(3)$, hence a compatible homomorphism to $\S^1$ will be
trivial on the face isotropy groups.
\end{rem}

\section{Non minimal  Grassmann Series for $\G_q$ slice
representation}\label{grassmann}

Recall that there are three infinite families of cases $(1,1,k), k
\ge 1$, $(2,2,2k+1), k \ge 1$ and $(4,4,4k+3), k \ge 0$ corresponding the real, complex and quaternion Grassmann series for the $\G_q$ slice
representation.

We point out that $(1,1,1)$ is special in two ways: There are two scenarios. One of them corresponding to the ``Flip" case of $d=1$ not covered in the previous subsection, the other being ``standard". Yet the standard $(1,1,1)$ does not appear as a reduction in any of the general cases $(1,1,k), k \ge 2$. For the $(2,2,3)$ case, there are two scenarios as well, both with \emph{the same local data}(!): One of them belonging to the family, the other not. Moreover, each of the cases $(2,2, 2k+1)$ with $k\ge 2$ admit a reduction to the ``Flip" $(2,2,1)$ case, whereas $(2,2,3)$ does not.

\smallskip

For the reasons just provided, this subsection will deal with the multiplicity cases $(1,1,k), k
\ge 2$, $(2,2,2k+1), k \ge 2$ and $(4,4,4k+3), k \ge 0$, each of which has a uniform treatment.

\smallskip

Although the case $(2,2,3)$ is significantly different from the other
general cases to be treated here, we begin by pointing out some
common features for all the cases $(1,1,k), k \ge 2$, $(2,2,2k+1), k \ge 1$
and $(4,4,4k+3), k \ge 0$, i.e., including the case $(2,2,3)$.

\medskip

To describe the information we have about the local data in a uniform fashion, we
 use $\G_d(k)$  to denote $\SO(k)$, $\SU(k)$ and
$\Sp(k)$, $k \ge 1$, according to $d=1$, $2$ and $d=4$, with the
exceptional convention that  $\G_1(-1)=\Bbb Z_2$, $\G_2(-1)=\S^1$ or
$\T^2$, depending on whether the center of $\K_t$ is finite or not,
and $\G_4(0)=\G_4(-1)= \{1\}$. Also, we use the symbol "=" to mean
"isomorphic" up to a finite connected covering.

\begin{lem}\label{vertex} In all cases $\G_t$ is connected as are $\G_q$ and $\G_r$ when $d \ne 1$. Moreover $\K_t=\G_d(k)$ with the additional
possibility that $\K_t = \G_d(k)\cdot \S^1$ when $d=2$.

 For the $q$ and $r$ vertex isotropy groups we have:  $\G_q=\G_d(2)\G_d(k+2)\cdot \G_d(-1)$, $\G_r=\G_d(2)\G_d(k+1)\cdot \G_d(-1)$. Moreover,
 the normal subgroup
$\K\lhd \G_r$ is $\G_d(k+1)\cdot \G_d(-1)$,  where $\G_d(k+1)$ is a
block subgroup of $\G_d(k+2)\lhd \G_q$, and if $d=1$,  ``$\cdot
\G_d(-1)$" denotes a nontrivial extension. In particular, $\G_q=\S(\O(2)\O(k+2))$.
\end{lem}

\begin{proof} The connectedness claim is a direct consequence of transversality. The proof follows the same strategy in all cases, just simpler when all vertex isotropy groups are conneceted. The two possibilities for $\G_t$ when $d=2$ correspond to the different rank possibilities for $\bar \G_q$, cf. Table \ref{q-rep}. For these reasons we only provide the proof in the most subtle case of $d = 1$.

First, notice that the effective slice representation $\bar \G_t=
\SO(3)$ on $T_t^\perp$ is of type $\A_2$ with principal isotropy
group $\bar \H=\Bbb Z_2^2$. Hence, $\H$ is an extension of $\Bbb
Z_2^2$ by the kernel $\K_t$. On the other hand, $(\bar \G_q)_0=
\SO(2)\SO(k+2)$  (cf. Table  \ref{q-rep}), and $\bar \G_q\subset
\O(2)\O(k+2)$, up to a possible quotient by a diagonal $\Bbb Z_2$ in
the center if $k$ is even.  Therefore, $\H$ is also an extension of
$\SO(k)$, $\SO(k)\cdot \Bbb Z_2$ or $\SO(k)\cdot \Bbb Z_2^2$ by
$\K_q$. This together with Lemma \ref{effectiveC3}, implies that
$\K_t =\SO(k)$ and hence $\G_t=\SO(3)\SO(k)$.
In particular, $\H=\SO(k)\times \Bbb Z_2^2$.

We conclude that $\G_{\ell_r}=\O(2)\SO(k)$, and similarly,
$\G_{\ell_q}=\O(2)\SO(k)$, acting on the normal sphere $\Sph^1$ with
principal isotropy group $\H$. Thus $\K_{\ell _r}=\SO(k)\times \Bbb
Z_2=\K_t\times \Bbb Z_2$. Since $\K_q\lhd \K_{\ell _r}$, we get
easily that $\K_q=\{1\}$ or $\Bbb Z_2$, since $\K_q\cap \K_t=\{1\}$.
On the other hand, as a subgroup of $\G_q$,
$\G_{\ell_r}=\Delta(\O(2))\SO(k)$. Hence $\G_q$ contains exactly two
connected components, whose identity component is
$\SO(2)\SO(k+2)\supset (\G_{\ell_r})_0$. All in all it follows that, $\G_q=\S(\O(2)\O(k+2))$. The rest of the proof is
straightforward.
\end{proof}

Note that  $\K_t$ contains $\G_d(k)$ as a normal subgroup. The fact that the reduction $M^{\G_d(k)}$ with the action by the identity component of its normalizer, $\N_0(\G_d(k))$, will give a geometry of type $\A_3$ or
 $\CC_3$ will play an important role in the $d=1,2$ cases below (cf. \ref{rank3crit}).

\medskip

 In what follows we will consider the reduction $M^{\K'}$ by  $\K'= \G_d(k+1) \lhd \G_d(k+1)\cdot \G_d(1) = \K \subset \G_r$ rather than the one by $\K$.

\begin{lem} The cohomogeneity one $\N(\K')$ manifold $M^{\K'}$ has  multiplicity pair  $(d, 2d-1)$, and the action is not equivalent to the reducible cohomogeneity one action on $\Sph^{2d-1} * \Sph^{d}$.
\end{lem}

\begin{proof}
For simplicity we give a proof for $d=2$, all other
cases are the same.

First note that the orbit space of the cohomogeneity one
$\N(\K')$-action is $\overline{rq}$, and the two singular isotropy
groups (mod kernel) are $\SU(2)\cdot\S^1$ and
$\SU(2)\cdot \T^2$ respectively, with principal isotropy group $\T^2$.
Hence the multiplicity pair is $(2,3)$.

To prove that it is not reducible, we argue by contradiction. Indeed, if
$M^{\K'}$ is equivariantly diffeomorphic to $\Sph^2*\Sph^3$
with the product action of $\SU(2)\U(2)$, it follows that the normal subgroup
$\SU(2)\lhd \G_q$ is also normal in $\N(\K')$. By primitivity
$\G=\langle \G_r, \G_q\rangle =\langle \N, \G_q\rangle$ and hence $\SU(2)$
is normal in $\G$. On the other hand, the face
isotropy group $\G_{\ell _r}\subset \G_t$ contains a subgroup
$\SU(2)$ which sits as $\Delta(\SU(2))\subset \G_q$. Therefore, the
projection homomorphism $p: \G\to \SU(2)$ is an epimorphism on
$\Delta(\SU(2))$. However, since it sits in $\SU(3)\lhd \G_t$ it must be
trivial, because any homomorphism from $\SU(3)$ to
$\SU(2)$ is trivial. A contradiction.
\end{proof}

When $d=1$ this is not immediately of much help since there are several positively curved irreducible cohomogeneity one manifolds with multiplicity pair $(1,1)$ (cf. Tables A and E in \cite{GWZ}) whose associated chamber system is not  of type $\CC_2$. However, when $d=2$, respectively $d=4$ corresponding to multiplicity pairs $(2,3)$, respectively $(4,7)$ we read off from the classification in \cite{GWZ} that

\begin{cor}\label{1reduction} The universal covering of $M^{\K'}$ is
 equivariantly diffeomorphic to a linear action of type $\CC_2$ on $\Sph^{11}$,
$\C\Bbb P^5$ or $\Bbb{HP}^2$ when $d=2$, and on $\Sph^{23}$ when $d=4$.
 \end{cor}

We are now ready to deal with each family individually, beginning with $d=1$, i.e. with the standard $(1,1,k\ge2)$ case, where the (almost) effective slice representation at $q \in Q$ is  the
defining tensor product representation of $\SO(2)\SO(k+2)$.

\begin{prop} \label{thmC3a}
In the standard $(1,1,k)$ case with $k \ge 2$, the associated
chamber system $\mathscr{C}(M;\G)$ is a building, with the isotropy representation of $\SO(k+3)/\SO(3)\SO(k)$ as a linear model.
\end{prop}

\begin{proof} By Lemma \ref{vertex} $\K_t=\SO(k)$, which is a
normal subgroup of the principal isotropy group $\H$. Consider the
reduction $M^{\K_t}$ with the action of its normalizer $\N(\K_t)$,
once again a polar action with the same section $\Sigma$. By Lemma
\ref{vertex} it is clear that the identity component of
$\N(\K_t)\cap \G_q$ is $\T^2$. Hence, the subaction by $\N_0(\K_t)$,
the identity component of $\N(\K_t)$, is of type $\A_3$, with a
right angle at $q$. Therefore, from the classification of $\A_3$
geometries  
(cf. section 7 in \cite{FGT}) it is immediate that, the universal cover
of $M^{\K_t}$ is equivariantly difffeomorphic to $\Sph^8$ with the
linear action of $\SO(3)\SO(3)$. In particular, if the section
$\Sigma = \Sph^2$, then $M^{\K_t}=\Sph^8$ and the chamber complex
for the subaction is a building of type $\A_3$, and we are done by 
 Remark \ref{rank3crit}, since Property   (P) is clearly satisfied for $\K = \SO(k+1) \lhd \G_r$.

It remains to prove that $M^{\K_t}$ is simply connected. Consider
the normal subgroup $\SO(2)\lhd \G_q$, and the fixed point component
$M^{\SO(2)}$,  a homogeneous manifold of positive curvature with
dimension at least two, since $M^{\SO(2)}\cap M^{\K_t}\subset
M^{\K_t}$ is of dimension $2$. Since the identity component of the
isotropy group, $(\G_q)_0=\SO(2)\SO(k+2)$, we see that
$M^{\SO(2)}=\Sph^{k+2}$ or $\Bbb {RP}^{k+2}$, according to
$M^{\K_t}\cap M^{\SO(2)}=\Sph^2$ or $\Bbb{RP}^2$, equivalently,
according to $M^{\K_t}=\Sph^8$ or $\Bbb{RP}^8$. We argue by
contradiction. If $M^{\SO(2)}=\Bbb {RP}^{k+2}$, then the identity
connected component of the normalizer $\N(\SO(2))$ acts transitively
on it with principal isotropy group $\SO(2)\O(k+2)\subset \G_q$.
Hence $\G_q=\SO(2)\O(k+2)$, a contradiction, since
$\G_q=\S(\O(2)\O(k+2))$.
\end{proof}

\smallskip

\begin{prop} \label{thmC3d=2a} In the standard $(2,2,2k+1)$ case, with  $k\ge 2$ the chamber system $\mathscr{C}(M;\G)$ is covered by a building, with the isotropy representation of $\U(k+3)/\U(k)\U(3)$ as a  linear model. 
\end{prop}

\begin{proof}

First note that the reduction $M^{\SU(k)}$, where $\SU(k) \lhd \K_t$, $k \ge 2$, is a positively curved cohomogeneity two manifold of type $\CC_3$ with multiplicity triple $(2,2,1)$. Moreover, $\SU(k)$ is a block subgroup in $\K' \subset \K$, where $\K'=\SU(k+1) \subset \SU(k+2) \lhd \G_q$ and of course $M^{\K'} \subset M^{\SU(k)}$.

We will prove that both reductions above are simply connected, by appealing to the Connectivity 
Lemma \ref{connect} of Wilking \cite{Wi3}.
To do this we now proceed to prove that codim$M^{\K'} \subset M^{\SU(k)} = 6$, and codim$M^{\SU(k)} \subset M = 6k$.

By the Spherical isotropy Lemma \ref{WilkC3}, every irreducible isotropy subrepresentation of
$\K'=\SU(k+1)$ is the defining representation $\mu _{k+1}$. From Table B in
[GWZ] and the above fact that $\SU(k+2)  \supset \K'$ it follows that, there is a simple normal
subgroup $\L\lhd \G$ such that, $\SU(k+2)\lhd \G_q$ projects to a
block subgroup of $\L$ where $\L= \SU(n)$ if $k\ge 4$, $\L=\SU(n)$
or $\SO(n)$ if $k=3$, and finally $\L=  \SU(n), \SO(n)$ or one of the
exceptional Lie groups $\F_4\subset \E_6\subset \E_7\subset \E _8$,
if $k=2$.

On the other hand, by the Flip Proposition \ref{flip2} the normalizer
$\N(\SU(k))$ is either $\SU(3)\SU(3)$ or $\U(3)\SU(3)$ modulo $\K_t$.
Since $\SU(k)\lhd \K_t$ is a block subgroup in $\K'$, this
together with the above implies that in fact $\L=\SU(k+3)$ for all $k\ge 2$, and only one such factor exist.
In particular, the $\K'$-isotropy representation along $\ell _t$ contains
exactly $3$ copies of $\mu_{k+1}$, one copy along the normal slice
$T_{\ell _t}^\perp$, and two copies along the orbit $\G/\G_{\ell_t}$.
Therefore, the codimension of $M^{\K'}$ in $M$ is $6(k+1)$, and
hence the codimension of $M^{\SU(k)}$ in $M$ is $6k$. By the Connectivity
lemma \ref{connect} of Wilking, we conclude that $\pi _i(M)\cong \pi
_i(M^{\SU(k)})$ for $i\le 2$, by induction on $k$. In particular,
$M^{\SU(k)}$ is simply connected and hence $\Sph^{17}$ if dim$(M)$ is odd and $\Bbb{CP}^8$ if dim$(M)$ is even, by the Flip Proposition \ref{flip2}. Since  Assumption  (P) in Lemma
\ref{Criterion2} is satisfied we conclude from \ref{rank3crit} that  $\mathscr{C}(M;\G)$ is a
building if dim$(M)$ is odd.

It remains to prove that $\mathscr{C}(M;\G)$ is covered by a
building if dim$(M)$ is even. In this case,  by the above we know
that $\pi _2(M)\cong \pi _2(M^{\SU(k)}) \cong \Bbb Z$. On the other
hand, from the Transversality Lemma \ref{transv}  it follows that $\pi_2(M) \cong
\pi _2(\G/\G_t)$, and hence $\G_t$ contains at least an $\S^1$  in its
center, i.e, $\SU(3)\U(k)\lhd \G_t$. By Lemma \ref{vertex} we get
that, both $\G_q$ and $\G_r$ have at least a $\T^2$ factor, and we are now in the
same situation as in the proof of Lemma \ref{flip2} above. As a consequence we can
proceed with the same construction of a principal $\S^1$ bundle $P$ over $M$ and conclude that its associated chamber system is a building covering  $\mathscr{C}(M;\G)$.
\end{proof}

\begin{prop} \label{t} In the standard $(4, 4, 4k+3)$ case where $k \ge
0$, the chamber system $\mathscr{C}(M;\G)$ is a building, with the isotropy representation of $\Sp(k+3)/\Sp(k)\Sp(3)$ as a linear model. 
\end{prop}

\begin{proof}
Since the Assumption (P) for $\K' = \Sp(k+1)$ in Lemma
\ref{Criterion2} is easily seen to be satisfied, it suffices by Corollary
\ref{1reduction} to prove that $M^{\K'}$ is simply
connected. As in the proof of the general $(2,2,2k+1)$ case above this is achieved via Wilkings Connectivity Lemma \ref{connect}.

Consider the normal subgroup $\Sp(2)\lhd  \G_q$. It is
clear that $M^{\Sp(2)}$ is a homogeneous space with a transitive
action by the identity component of its normalizer $\N_0(\Sp(2))$ with isotropy group $\G_q$. By
the classification of positively curved homogeneous spaces 
we get
that $M^{\Sp(2)}$ is either $\Sph^{4(k+3)-1}$ or $\Bbb
{RP}^{4(k+3)-1}$. Moreover, the universal cover $\tilde{\N_0}(\Sp(2))$
is $\Sp(k+3)\Sp(2)\Sp(1)$, and in particular has the same rank as $\G$ by the Rank
Lemma.

On the other hand, by Lemma \ref{WilkC3} and Table  B in [GWZ] it
follows that, $\G$ contains a normal subgroup isomorphic to $\Sp(n)$
so that $\K'\subset \Sp(k+2)\subset \Sp(k+3)\subset \Sp(n)$ is in a
chain of block subgroups. Up to a finite cover, we let
$\G=\Sp(n)\cdot \L$. On the other hand, by Corollary \ref{1reduction} we
know that $\N_0(\K')=\Sp(2)\Sp(3)\K'$. This together with the
information on $\tilde{\N}(\Sp(2))$ implies that $\G=\Sp(k+3)\cdot
\L$. As in the proof of the $(2,2,2k+1)$ case we see that the isotropy representation of
$\K'$, along $\ell _t$ contains exactly three copies of $\nu_{k+1}$,
one copy along the normal slice $T_{\ell _t}^\perp$, and two copies
along the orbit $\G/\G_{\ell_t}$. In particular, the codimension of
$M^{\K'}$ in $M$ is $12(k+1)$. Recalling that the dimension of
$M^{\K'}$ is $ 23$, it follows again by connectivity and induction on $k$ as before that $M^{\K'}$ is simply connected.
\end{proof}

\section{Minimal  Grassmann $\G_q$ slice representation}\label{minimal}

This section will deal with the multiplicity cases $(1,1,1)$ and $(2,2,3)$, including the appearance of an exceptional Cayley plane action. In all previous cases all reductions considered have been irreducible polar actions. Here, however, we will encounter reductions, that are \emph{reducible cohomogeneity two actions}, and we will rely on the independent classification of such actions in sections 6 and 7 of \cite{FGT}.

\smallskip

We begin with the $d=2$ case, where by \ref{1reduction} we know that the universal covering
$\tilde{M}^{\K'}$ of the reduction $M^{\K'}$ is   diffeomorphic to $\Sph^{11}$,
$\Bbb {CP}^5$ or $\Bbb{HP}^2$. The first two scenarios follow the outline of the general $(2,2,2k+1)$ case, whereas  the latter is significantly different.

\begin{prop}
\label{lemC3d=2,3a} In the case of multiplicities $(2, 2, 3)$, $\mathscr{C}(M;\G)$ is covered by a
building, with the isotropy representation of $\U(7)/\U(4)\U(3)$ as a linear model, provided
$M^{\K'}$ is  not  diffeomorphic to $\Bbb {HP}^2$.
\end{prop}

\begin{proof} By Lemma \ref{vertex}, $\G_t$ is either $\SU(3)$ or $\U(3)$ depending on whether $\K_t$ is finite or $\S^1$. In the latter case, the reduction $M^{\K_t}$ is a positively curved cohomogeneity two manifold of type $\CC_3$ with multiplicity triple $(2,2,1)$, as
in the general $(2,2, 2k+1)$ case, where $k\geq 2$ (cf. \ref{thmC3d=2a}). Therefore, $\N_0(\K_t)/\K_t= \SU(3)\cdot \SU(3)$ or $\U(3)\cdot \SU(3)$, by the Flip Proposition \ref{flip2}. The desired result follows, as in the proof of Proposition \ref{thmC3d=2a}.

From now on we assume that, up to finite kernel, $\G_t=\SU(3)$, and correspondingly, $\G_q=\U(2)\SU(3)$, and $\G_r=\U(2)\SU(2)$. Moreover, $\K'=\SU(2)$, and from our assumption on the reduction $M^{\K'}$, by Corollary \ref{1reduction} the normalizer $\N(\K')$ contains $\SU(2)\SU(2)\SU(3)$ as its semisimple part. On the other hand, by the Rank Lemma \ref{rank} we
know that $\text{rk}(\G)=5$ (resp. $\text{rk}(\G)=4$) if dim$(M)$ is
odd (resp. even). In particular, $\SU(2)\SU(2)\SU(3)$  is a maximal rank subgroup of
$\G$ if $\text{rk}(\G)=4$. In this case, it is immediate, by Borel
and de Siebenthal \cite{BS} (see the Table on page 219),  that $\G$ is not a simple group of
rank $4$. Similarly, we claim that $\G$ is not a simple group when its rank is $5$: Indeed if so, by Lemma
\ref{WilkC3} and Table B in [GWZ], it would follow that $\G=\SU(6)$ and
$\K'=\SU(2)\subset \SU(3)\lhd \G_q$ is a block subgroup. This, however, is not possible, since then $\N(\K')$ would contain $\SU(4)$.
Thus, $\G=\L_1\cdot \L_2$, where $\L_1, \L_2$ are nontrivial
Lie groups. Without loss of generality, we assume that the projection of $\SU(3)\lhd
\G_q$ to $\L_2$ has nontrivial image. But then
$\SU(3)$ must be contained in $\L_2$, because otherwise, the
normalizer $\N(\K')$ would be much smaller than $\SU(2)\SU(2)\SU(3)$.
By Primitivity \ref{prim} it is easy to see that $\G_t$ is diagonally imbedded
in $\L_1\cdot \L_2$, since $\G=\langle \G_t, \G_{\ell _t}
\rangle=\langle \G_t, \K'\rangle$. In particular, both $\L_1$ and
$\L_2$ have rank at least two since the projections from $\G_t$ are
almost imbeddings, i,e,, have finite kernel.  If both $\L_1$ and $\L_2$ have rank two,
 it is easy to see that, $\L_1=\SU(3)$ and $\K'\subset \L_2$, where $\L_2=\SU(3)$ or $\G_2$. Neither scenario is possible: For the latter since,  by the primitivity, $\G=\langle \Delta (\SU(3)), \K'
\rangle= \SU(3)\cdot \SU(3)$, while for the former the semisimple part of $\N(\K')$ is $\L_1$.  Therefore $\text{rk}(\G)=5$ and once again by Lemma \ref{WilkC3} and
Table B in [GWZ], $\G=\SU(3)\SU(4)$.

Note that dim$M = 21$ and the principal orbit of $\K'$ in $M$ is of dimension at least $2$. In particular, it follows from Wilkings Connectivity Lemma\ref{connect} that  $M^{\K'}$ is
simply connected. Thus, as in the general case the desired result follows from Lemma \ref{Criterion2}.
\end{proof}

\begin{prop} \label{lemC3d=2,3b} In the case of multiplicities $(2, 2, 3)$, $M$ is equivariantly diffeomorphic to the Cayley plane
$\Bbb O\Bbb P^2$ with an isometric polar action by $\SU(3)\cdot\SU(3)$, provided  $M^{\K'}$ is diffeomorphic to $\Bbb {HP}^2$.
\end{prop}

\begin{proof}
Recall that $\K'=\SU(2)\lhd \G_{\ell_t}$. By Lemma \ref{WilkC3} and the slice representation of $\G_{\ell_t}$ it follows that, every irreducible subrepresentation 
of $\K'$ on the normal space to
$M^{\K'}$ is the standard representation $\mu_2$ on $\Bbb C^2$. In particular, the codimension of $ M^{\K'}$ is a multiple of $4$, and so $M$ has dimension divisible by $4$.
By \ref{vertex} the isotropy group $\G_t=\SU(3)$ or $\U(3)$, and correspondingly, $\G_q=\U(2)\SU(3)$ or $\U(2)\U(3)$, and $\G_r=\U(2)\SU(2)$ or $\U(2)\U(2)$. By the Rank Lemma $\text{rk}(\G)=\text{rk}(\G_q)=4$ or $5$.

By Lemma \ref{WilkC3} the isotropy representations of $\K'\subset \SU(3)\lhd \G_q$,  as well as of $\SU(2)\subset \G_{\ell_q}\subset \G_t$, are spherical transitive. By Table B in [GWZ] it follows that, $\G$ can not be a simple group of rank $5$, and moreover, $\G$ can not contain $\F_4$, $\Sp(4)$, $\SO(8)$ and $\SO(9)$ as a normal subgroup, since if so, the semisimple part of $\N_0(\K')/\K'$ would not be $\SU(3)$, a contradiction to our assumption on the reduction $M^{\K'}$, for which $\N_0(\K')/\K'=\SU(3)\cdot \S^1$.  On the other hand, note that the identity component of the normalizer $\N_0(\G_t)=\G_t$ since $\G_t$ is a maximal isotropy group and hence $\N_0(\G_t)/\G_t$ acts freely on the positively curved fixed point set $M^{\G_t}$ of even dimension. Therefore, $\G$ can not contain $\SU(5)$ as a normal subgroup, since otherwise, $\G_t$ would be a block subgroup in $\SU(5)$ and hence $\N_0(\G_t)/\G_t$ would not be trivial. Consequently, $\G$ is not a simple group, and moreover, $\G=\L_1\cdot \L_2$, where $\SU(3)\lhd \G_t$ is diagonally imbedded in $\G$. In particular, 
both $\L_1$ and $\L_2$  contain $\SU(3)$ as subgroups.   It is easy to see that, $\SU(3)\lhd \G_q\subset \G=\L_1\cdot \L_2$ is
a subgroup in either $\L_1$ or $\L_2$, say in $\L_2$. Hence, $\K'\subset \L_2$, and $\L_1\lhd \N(\K')$. It follows that $\L_1=\SU(3)$. Furthermore, $\L_2$ can neither be a semi-simple group of rank $3$ or $\G_2$, since otherwise, $\N_0(\K')/\K'$ contains a rank $3$ semisimple group. Hence, $\L_2$ is $\SU(3)$ or $\U(3)$. The latter, however, is impossible: Indeed, in this case  $\G_t=\U(3)$, and the center $\S^1\subset Z(\G)$ would be contained in $\K_t$, and hence in every principal isotropy groups (the center is invariant under conjugation) thus $M^{\S^1} = M$.

In summary  we have proved that $\G=\SU(3)\cdot \SU(3)$ (indeed a quotient group by $\Delta (\Bbb Z_3)$), with $\G_t=\SU(3)$ diagonally imbedded in $\G$. We claim that this combined with the above analysis of the isotropy groups modulo conjugation will force the polar data $(\G_t, \G_q, \G_r) \subset \G$ (noting that face isotropy groups are intersections of vertex isotropy groups) to be  $(\G_t, \G_q, \G_r)=(\Delta (\SU(3)), \U(2)\cdot \SU(3), \S(\U(2)\U(2)))$, where $\U(2)\subset \SU(3)$ is the upper $2\times 2$ block subgroup in $\SU(3)$, and $\S(\U(2)\U(2)) \subset \SU(3)\cdot \SU(3)$ is the product of the lower $2\times 2$ block subgroups. In other words, by the recognition theorem for polar actions \cite{GZ} there is at most one such polar action. - On the other hand the unique action by the maximal subgroup $\SU(3)\cdot\SU(3) \subset \F_4$, the isometry group of the Cayley plane $\Bbb O\Bbb P^2$ is indeed polar of type $\CC_3$ \cite{PTh}.

To prove the above claim, by conjugation we may assume that  $\G_t=\Delta (\SU(3))$ and $\G_q=\U(2)\cdot \SU(3)$ as claimed. Moreover, up to conjugation by an element of the face isotropy group $\G_{\ell_r}=\G_t\cap \G_q$, we may further assume that $\K'\lhd \G_{\ell_t}\subset \G_q$ is the lower $2\times 2$ block subgroup in the second factor $\SU(3)\lhd \G$.
Note that $\K'$ is a normal subgroup of $\G_r$, indeed the second factor of $\SU(2)\cdot \SU(2)\lhd \G_r \subset \SU(3)\cdot\SU(3)$.  Since $\G_{\ell_q}=\Delta(\SU(3))\cap \G_r$, it follows that $\SU(2)\cdot \SU(2) \lhd \G_r$ is the product of the lower $2\times 2$ block subgroups. Since  $\G_r=\langle \SU(2)\cdot\SU(2), \H\rangle$ where $\H=\Delta (\T^2)$ is the principal isotropy group, the desired assertion follows.

\end{proof}

 Next we deal with the case of multiplicity $(1, 1, 1)$, where there are two scenarios: One is naturally viewed as part of the infinite family $(1,1,k)$, whereas the other should be viewed as the flip case with $d=1$.

 We point out that unlike all other cases an $\S^3$ chamber system cover arises in the first case, corresponding to a polar action of $\SO(3)\SO(3)$ on $ \Bbb {HP}^2$.

\begin{prop}
For the multiplicity $(1, 1, 1)$ case, the chamber system $\mathscr{C}(M;\G)$ is covered by a building, with the isotropy representation of either $\SO(7)/\SO(4)\SO(3)$, or of  $\Sp(3)/\U(3)$ as a linear model.
\end{prop}

\begin{proof}
Recall that, $\bar \G_t=\SO(3)$, and $\bar \H=\Bbb Z_2^2$. We first
claim that the identity component $(\G_t)_0=\SO(3)$. To see this,
recall that the kernel $\K_t\subset \K_{\ell_r}$, and
$\bar {\G}_q$ is either $\SO(2)\SO(3)$ or $\S(\O(2)\O(3))$.  The claim follows since, if
$\dim \K_t \ge 1$ or $(\G_t)_0=\S^3$, then  $\K_t\cap
\K_q$ is nontrivial, a contradiction to Lemma
 \ref{effectiveC3}. From this we also conclude that $(\G_q)_0$ is not $\S^1\times \S^3$, since otherwise again  $\K_t\cap \K_q$ is non-trivial. Hence it is isomorphic to either $\SO(2)\SO(3)$ (the ``standard" case)
or to the $2$ fold covering $\U(2)$ of $\SO(2)\SO(3)$ (the ``flip" case). By
 the Rank Lemma  \ref{rank} it follows that $\text{rank}  \G \le 3$.

We start with the following observation:

$\bullet$ Let $z$ be cyclic subgroup of the principal isotropy
group $\H$ with non-trivial image $[z]\subset \bar \H$. Then the action by $\N_0(z)$ on the reduction $M^{z}$ is a
\emph{ reducible} polar action of cohomogeneity $2$. To see this note that  the type $t$ orbit in the reduction is no longer a vertex. Indeed the normalizer of $[z] \subset \bar \H\subset \SO(3)$ is
$\O(2)$.

In addition, note that the identity component of every face
isotropy group is $\S^1$. By the Dual Generation Lemma 7.2 in \cite{FGT} we conclude that

\smallskip

$\bullet$ 7.3.1. {\it The semisimple part of $\N_0(z)$ has rank at
most one.}

\smallskip

To proceed we will prove that

(a).  $\G$ is not a simple group of rank $3$.

This is a direct consequence of  7.3.1  combined with the following
algebraic fact: If $\G$ is a rank $3$ simple group, i.e., one of
$\SO(6)=\SU(4), \SO(7)$ or $ \Sp(3)$ (up to center), then, the
normalizer of any order $2$ subgroup $\Bbb Z_2\subset \SO(3)=
(\G_{t})_0$ contains a semisimple subgroup of rank at least $2$. The
algebraic fact is easily established by noticing that the inclusion
map $\SO(3)\to \G$ either can be lifted to a homomorphism into one
of the four matrix groups, or $\SO(3)$ sits in the quotient image of
a diagonally imbedded $\SU(2)$ in one of the matrix groups.

 Next we are going to prove that

(b). If $\G$ is a rank $2$ group, then either $(M, \G)=(\Bbb {CP}^5,
\SU(3))$ or $(\Bbb{HP}^2, \SO(3)\SO(3))$ up to equivariant
diffeomorphism.

Exactly as in Case (a), we can exclude $\G$ being  $\SO(5)$
since a subgroup $\Bbb Z_2\subset \bar \H\subset (\G_q)_0$
will have a normalizer containing $\SO(4)$. We now exclude $\G$ being the exceptional group $\G_2$.
Otherwise, $(\G_q)_0$ must be $\U(2)$, and contained in either an $\SO(4)\subset \G_2$
or an $\SU(3)\subset \G_2$ by Borel-Siebenthal \cite{BS}. The center $\Bbb Z_2\subset \U(2)$ is in $\K_q$.  For the same
reason as above, $\U(2)$ is not in $\SO(4)\subset \G_2$. Finally, if $\U(2)\subset \SU(3)\subset \G_2$, the $q$ orbit $(\G q)^{\Z_2}$ in the reduction $M^{\Z_2}$ contains $(\G_2/\SU(3))^{\Z_2}=(\Sph^6)^{\Z_2} = \Sph^2$. Again by the Dual Generation Lemma 7.2 of \cite{FGT} this is impossible, since the identity component of the isotropy group of the face opposite of $q$ is a circle, which cannot act transitively on the orbit $(\G q)^{\Z_2}$. Therefore, up to local isomorphism, $\G$ is $\SO(3)\SO(3)$ or $\SU(3)$ respectively. One checks that the corresponding isotropy group data are given by $\G_t=\Delta
(\SO(3))\subset \SO(3)\SO(3)$, and $\G_q=\O(2)\SO(3)\subset
\SO(3)\SO(3)$, respectively by $\G_t=\SO(3)\subset \SU(3)$ (inclusion
induced by the field homomorphism), and $\G_q= \U(2)\subset \SU(3)$
as a block subgroup. The recognization theorem then yields (b).

(c). Now suppose $\G=\L_1 \cdot \L_2$, where  $\L_i$ is a rank $i$ Lie group.

If $\L_1$ acts freely on $M$, then $\L_1=\S^1$, $\SO(3)$, or $\S^3$, and $\L_2$
acts on $M/\L_1$ in a polar fashion of type $\CC_3$. Hence, $M/\L_1$
is even dimensional and thus $\Bbb {CP}^5$ or $\Bbb{HP}^2$ by (b). In either case, we know that the universal cover $\tilde {\mathscr{C}}$ of the
chamber system $\mathscr{C}(M/\L_1,\L_2)$ is a building. Since $\mathscr{C}(M,\G)$ is a connected chamber system covering $\mathscr{C}(M/\L_1,\L_2)$ it follows that
$\tilde {\mathscr{C}}$ is the universal cover of $\mathscr{C}(M,\G)$.

 Now consider the remaining case where

 $\bullet$ $\L_1$ does not act freely on $M$, and we let $\Bbb
Z_m\subset \L_1$ be a cyclic group such that $M^{\Bbb Z_m}\neq
\emptyset$.

 \smallskip
Note that $\G$ can not be $\SO(3)\cdot \T^2$, since, then $\G_t$ and $\SO(3)\lhd
\G_q$ would be 
the same simple group factor, which is absurd. In particular, the semi-simple part of $\G$ has rank at least two. Thus from now on we may assume that $\L_2$ is a rank two semi-simple 
group. Moreover, by the argument in Case (b) it is immediate that in fact $\L_2$ is either $\SO(4)$ or $\SU(3)$.

Notice that:

$\bullet$ If $\K_t$ is not trivial, then $(M^{\K_t}, \N_0(\K_t ))$
is a polar manifold with the same section, which is of  type $\A_3$.
By the Connectivity Lemma \ref{connect} it follows that $M^{\K_t }$ is simply
connected. Hence,  from the classification of $\A_3$ geometries,
$M^{\K_t}$ is diffeomorphic to $\Sph^8$, and the chamber system of
$(M^{\K_t}, \N_0(\K_t ))$ is a building. By \ref{rank3crit}
$\mathscr{C}(M,\G)$ is a building.

Therefore, we may assume in the following that $\K_t=\{1\}$, hence
$\G_t=\SO(3)$. It follows that, $\G_q$ is either $\S(\O(2)\O(3))$ or
$\U(2)$.

We split the rest of the proof according to $\L_1$ abelian or not.
In either case note that the normalizer $\N_0(\Bbb Z_m)$ is
$\S^1\cdot \L_2$. From this we get immediately that $\Bbb
Z_m\nsubset \H$, by appealing to  7.3.1.

(ci)   $\G=\S^1\cdot \L_2$.

It suffices to prove that    
the
$\S^1$ action is free, since then the situation
reduces to the
previous rank $2$ case.

Note that $\Bbb Z_m$ is normal in $\G$. From this and the above it
follows that $\Bbb Z_m$ is neither in $\G_t$ nor in $\G_{\ell_t}$.
To see this, if $\Bbb Z_m \subset \G_{\ell_t}$ then  $ (\bar
{\G}_q)_{\ell_t} \subset \bar \G_q$  would contain a non-trivial
normal subgroup of $\bar \G_q$ contradicting 
Table \ref{q-rep}. The
proof in the other case is similar but simpler. Hence,  $M^{\Bbb
Z_m}$ is either the orbit $\G r$ or $\G q$.

Assuming $M^{\Bbb Z_m} = \G/\G_r$, it is immediate that $\L_2 =
\SU(3)$ from the list of positively curved homogeneous spaces. On
the other hand, notice that $\G_r$ is not connected, indeed
$(\G_r)_0=\T^2$ and $\G_r\supset \G_{\ell_q}\supset \O(2)$, it
follows that $\G/\G_r$ is not simply connected.  However, $\G/\G_r$
is a totally geodesic submanifold in $M$ which has dimension $11$. A
contradiction to Wilking's Connectivity Lemma \ref{connect}.
 
 Assuming $M^{\Bbb Z_m} = \G/\G_q$,  corresponding to $\L_2=\SU(3)$
or $\SO(4)$, the universal cover of $\G/\G_q$ is a sphere of dimension either $5$ or $3$. The latter case is ruled out as follows: If $\G_q = \U(2)$ then $\K_q = \Z_2$ is in the center of $\U(2)$ hence also in the center of $\G$. This is impossible, since $\K_q \subset \H$ and $\G$ acts effectively on $M$ by assumption. If  $\G_q=\S(\O(2)\O(3))$ there are no non-trivial homomorphisms to $\S^1$, hence $\G_q\subset \SO(4)$, which is impossible. For
the former case, $\G_q=\U(2)$ and $\G=\U(3)$, with action on  $\G/\G_q$ equivalent to the standard
linear action on a 5-dimensional spherical space form with  $\Bbb
Z_m$ in the kernel. Thus, $\G_q \supset \Bbb Z_m\times
\U(2)$, a contradiction.

\smallskip

(cii)  $\G=\L_1 \cdot \L_2$, where $\L_1$ is a simple rank one group, i.e., either $\S^3$ or $\SO(3)$.

We will show that in this case $\G = \SO(3)\SO(4)$, with local data $\G_q=\S(\O(2)\O(3))\subset \G$, and $\G_t=\Delta (\SO(3))\subset\G$ forcing all data to
coincide with those of the isotropy representation of $\SO(7)/\SO(3)\SO(4)$, and hence $M$ with the action of $\G$ is determined via recognition.

We first prove that $\L_2=\SO(4)$. If not, we start with an observation that, $\L_1=\SO(3)$,
 and moreover, $\G_t$ is a diagonally imbedded
subgroup in $\L_1\cdot \L_2$. Indeed, otherwise, an order $2$
element $z\in \H\subset \G_t$ will have a normalizer $\N_0(z)$ which
contains a rank $2$ semisimple subgroup, contradicting 7.3.1.  For
the same reason, as above, we see that $\G_q \neq \U(2)$ and hence,
$\G_q=\S(\O(2)\O(3))$. Similarly by 7.3.1, $\SO(3)\lhd \G_q$ must
be diagonally embedded in $\L_1\cdot\L_2$.  This is impossible since
then $\N(\SO(3))/\SO(3)$ is finite, but $(\G_q)_0 \subset
\N(\SO(3))$.

Finally, given that $\L_2 = \SO(4)$ it follows as above that $\G_q \ne \U(2)$, hence $\G_q=\S(\O(2)\O(3))$.
Since $\G_t = \SO(3)$ and $\G_{\ell_r} = \O(2)$ sits diagonally in $\G_q$ it follows that $\G_t$ sits diagonally
in $\L_1\cdot \L_2$, in particular $\L_1 = \SO(3)$. Using the same arguments as above we see that $\SO(3)\lhd \G_q$ is in $\L_2$. All together,
all isotropy data are determined.
\end{proof}

\section{Exceptional $\G_q$ slice representation}\label{exceptional}

This section will deal with the remaining cases, all of which are exceptional with multiplicities $(1,1,5)$, $(2,2,2)$, $(4,4,5)$, and
$(1,1,6)$. All but the latter will occur, and the case of $(1,1,5)$ will include an exceptional action on the Cayley plane.

\smallskip

\begin{prop}\label{(1,1,5)} In the case of the multiplicities $(1, 1, 5)$ where the (effective) slice
representation at $T_q^\perp$ is the tensor representation of
 $\SO(2)\G_2$ on $\Bbb R^2\otimes \Bbb R^7$,  either $M$ is equivariantly diffeomorphic to the Cayley plane
$\Bbb O\Bbb P^2$ with an isometric polar action by $\SO(3)\cdot \G_2$ or $\mathscr{C}(M,\G)$ is
a building, with the tensor product representation of
$\SO(3)\Spin(7)$ on $\Bbb R^3\otimes \Bbb R^8$ as a linear model.
\end{prop}

\begin{proof}
By the Transversality Lemma \ref{transv} we conclude that $\G_t$ is connected since $\G/\G_t$ is simply connected. The kernel $\K_t$ is a normal subgroup in $\G_t$, as well as of the principal isotropy group $\H$ with quotients $\G_t/\K_t=\SO(3)$, and $\H/\K_t=\Bbb Z_2\oplus \Bbb Z_2$ respectively (cf. Table \ref{t-rep}). By the Slice Lemma \ref{effectiveC3} $\K_t$ acts effectively on the $q$-slice. Combining this with Table \ref{q-rep} where $(\bar \G_q)_0=\SO(2)\G_2$ it follows that, the identity components $(\K_t)_0=\H_0=\S^3$. Thus, $\G_t=\SO(4)$, or $\Spin(4)=\S^3\times \S^3$. The latter, however, is impossible, since then $\K_t=\S^3\times \Bbb Z_2\lhd \S^3\times \Q_8=\H$ where $\Q_8$ is the quaternion group of order $8$. On the other hand, by Table \ref{q-rep}
the slice representation at $q$ is the natural tensor representation of $\O(2)\G_2$ on $\Bbb R^2\otimes \Bbb R^7$, where the center $\Bbb Z_2\subset \Q_8$ is in the kernel $\K_q$ and so in $\K_t\cap \K_q$. A contradiction.  Therefore, $\G_t=\SO(4)$, and consequently, $\G_q=\O(2)\cdot \G_2$,  $\G_r=\O(2) \cdot \SU(3) $ and $\G_{\ell _t}= \SU(3)\cdot\Bbb Z_2^2$.

\vskip 2mm

By Lemma 6.3 we have $3\leq \text{rk}(\G)\le 4$.

Case (i). Assume $\text{rk}(\G)=3$:

By Lemma 6.3 again $\text{dim} M$ is even. By [BS] (table on page 219) $\SO(2)\G_2$ is not a subgroup in any rank $3$ simple group. Therefore, $\G=\L\cdot \G_2$, where $\L$ is a
rank one group.  By Table \ref{q-rep} the face isotropy group $(\G_{\ell _r})_0=\SU(2)\cdot \Delta(\SO(2)) $ is diagonally embedded in $\SO(2)\G_2\lhd \G_q\subset \L\cdot \G_2$. It follows that, the composition homomorphism $\G_t \subset \G\to \L$ is nontrivial, hence surjective onto $\L$, because $\G_t=\SO(4)$. Hence $\L=\SO(3)$ and $\G=\SO(3)\G_2$  since the only proper 
nontrivial normal Lie subgroup of $\SO(4)$ is $\S^3$ with quotient $\SO(3)$. By the above, we already know that, $\G_t=\SO(4)$ is a diagonal subgroup given by an epimorphism $\SO(4)\to \SO(3)$ and a monomorphism $\SO(4)\to \G_2$. It is clear that, up to conjugation, $\G_q=\O(2)\cdot \G_2\subset \G$ where $\O(2)\subset \SO(3)$ is the standard upper $2\times 2$ block matrices subgroup.

As in the proof of Proposition \ref{lemC3d=2,3b} we now claim that there is at most one polar action with the data as above. Since we are dealing with a non classical Lie group, however, we proceed as follows: 

 Given another $\CC_3$ type polar action of $\G=\SO(3)\G_2$ with isomorphic local data along a chamber $C'$ with vertices $t', q', r'$.
Without loss of generality we may assume that $\G_q=\G_{q'}\subset \SO(3)\G_2$, and moreover, $\G_t=\G_{t'}$ since any two $\SO(4)$ subgroups in $\G_2$ are conjugate. Moreover, we can further assume that $\G_{\ell_t}=\G_{\ell_{t'}}$ since the singular isotropy groups pair for the slice representation at $q$ is unique up to conjugation. In particular, the principal isotropy groups $\H=\H'$. We prove now $(\G_{\ell_q})_0=(\G_{\ell_{q'}})_0=\SO(2)\SU(2)$.  This clearly implies the assertion since $\G_r$ is generated by $(\G_{\ell_q})_0$ and $\G_{\ell_t}$. Recall that $\G_t=\Delta(\SO(4))\subset \G$, its composition with the projection to $\G_2\lhd \G$ is a monomorphism, so is the composition of $(\G_{\ell_q})_0\subset (\G_r)_0=\SO(2)\cdot \SU(3)$ to $\G_2$,  hence, $(\G_{\ell_q})_0$ is a diagonal subgroup of $\G_r$, whose projection to the factor $\SU(3)$ is injective.
Hence it suffices to show that the projection images of $(\G_{\ell_q})_0$ and $(\G_{\ell_{q'}})_0$ in $\SU(3)=(\G_{\ell_t})_0=(\G_{\ell_{t'}})_0$ coincide.
On the other hand, note that the projection image of $(\G_{\ell_q})_0$ in $\SU(3)$ is the normalizer $\N_0(\H_0)$ in $\SU(3)=(\G_{\ell_t})_0$, where $\H_0$ is the identity component of the principal isotropy group. The above assertion follows.

As for existence we again note that $\SO(3)\G_2$ is a maximal subgroup of the isometry group  $\F_4$ of the Cayley plane $\Bbb O\Bbb P^2$. The corresponding unique isometric action is indeed polar as proved in \cite{GK} and of type  $\CC_3$.

\vskip 2mm

Case (ii). Assume $\text{rk}(\G)=4$:

By Lemma \ref{rank}, $\text{dim} M$ is odd.  Consider the reduction $M^{\H_0}$   with the action
of $\N_0(\H_0)$, the identity component of the normalizer. Note that, this is also a $\CC_3$ type polar action,  but the multiplicity triple is $(1, 1, 1)$.  By
appealing to Lemma \ref{WilkC3}, the codimension of $M^{\H_0}$ is divisible by $4$. Thus from the $(1, 1, 1)$ case it follows that,
the universal cover $\tilde M^{\H_0}$  is $\Sph^{11}$, and the identity component
$\N_0(\H_0)$ is either $\U(3)$ or $\SO(3)\SO(4)$, modulo kernel.

We are going to prove that  $M^{\H_0}$ is simply connected. It
suffices to show that $M^{\H_0}\subset M$ is $2$-connected. This
follows trivially by the Connectivity Lemma of Wilking \ref{connect}, if the
codimension of $M^{\H_0}$ is at most $12$.

If $\G_2\lhd \G_q$ is a normal subgroup of $\G$,
then $\G=\L\cdot \G_2$ where $\L$ is a rank $2$ group. Then $\N_0(\H_0)/\H_0$ is isomorphic to $\L\cdot \SO(3)$.
Hence $\L=\SO(4)$. It is easy to count the codimension to see that it is strictly less than $12$.

If $\G_2$ is not a normal subgroup, by Lemma \ref{WilkC3} the
isotropy representation of $\SU(3)\subset \G_2\subset \G$ is
spherical transitive. Hence, $\G$ contains a normal simple Lie subgroup $\L$,
such that $\G_2\subset \Spin (7)\subset \L$ is \emph{spherical}. We claim that
$\L=\Spin(7)$. If not, $\L$ contains $\Spin(8)$
such that $\Spin(7)\subset \L$ is a block subgroup in $\Spin(8)$,
and hence $\N_0(\H_0)$ contains $\Spin (5)$, which contradicts the above. This proves that $\G=\L_1\cdot \Spin(7)$, where $\L_1$ is
a rank $1$ group. From this we get that the
isotropy subrepresentation of $\G/\H_0$ contains exactly three
copies of the standard defining representation of $\SU(2)$, hence the
desired estimate for the codimension.

In summary we conclude that $M^{\H_0}=\Sph^{11}$,
$\N_0(\H_0)=\SO(3)\SO(4)$ and hence, from the multiplicity $(1, 1,
1)$ case, the chamber system for the action of $\N_0(\H_0)$ is a
building of  type $\CC_3$. By remark \ref{rank3crit} we conclude
that $\mathscr{C}(M,\G)$ is a building.
\end{proof}

\medskip

\begin{prop}
There is no polar action of type $\CC_3$ type with multiplicities $(1, 1,
6)$, where the (effective) slice representation at $T_q^\perp$ is the tensor
product representation of $\SO(2)\Spin(7)$ on $\Bbb R^2\otimes \Bbb R^8$.
\end{prop}
\begin{proof}
We will prove that, if there is such a slice representation at $q$,
the chamber system  $\mathscr{C}(M,\G)$ is a building. The desired
claim follows from the classification of $\CC_3$ buildings, i.e.,
indeed there is no such a building.

To proceed, note that from Table \ref{q-rep} $\bar
\G_q=\SO(2)\Spin(7)$, and the principal isotropy group $\bar
\H=\SU(3)$. It follows that, up to local isomorphism
$\G_t=\SU(3)\SO(3)$ with $\K_t=\SU(3)$. Notice that, the reduction
$(M^{\K_t}, \N_0(\K_t))$ is of cohomogeneity $2$ with the same
section. It is clear that it is of type $\A_3$ since the $q$ vertex
is a vertex with angle $\pi/2$, because $\N_0(\K_t)\cap \G_q$ is
$\T^2$. By the classification of $\A_3$ geometries it follows that,
$M^{\K_t}$ is either $\Sph^8$ or $\Bbb{RP}^8$. We claim that
$M^{\K_t}=\Sph^8$, and hence the chamber system for $(M^{\K_t},
\N_0(\K_t))$ is a building. By appealing to \ref{rank3crit} it
follows that $\mathscr{C}(M,\G)$ is a building. To see the claim, it
suffices to prove that $M^{\K_t}$ is orientable and hence simply
connected, thanks to the positive curvature. By \ref{WilkC3} the
isotropy representation of $\K_t=\SU(3)$ is the defining complex
representation. From this it is immediate that, $M^{\K_t}=M^{\T^2}$,
and hence oriented, where $\T^2\subset \K_t$ is a maximal torus.
\end{proof}

\smallskip

\begin{prop}
When the multiplicity triple is $(2, 2, 2)$, there are two scenarios. In either case
$\mathscr{C}(M,\G)$ is a building, with linear model the adjoint
polar representation of either $\SO(7)$ or of $\Sp(3)$ on $\Sph ^{20}$. 
\end{prop}

\begin{proof}
By Lemma
 \ref{transv} we know that all vertex isotropy groups are
connected. Notice that, by Table \ref{q-rep},  the slice
representation at $q$ is the adjoint representation of $\SO(5)$  on
$\Bbb R^{10}$.  Together with Proposition \ref{effectiveC3}, up to
local isomorphism, the local isotropy group data are determined as
follows:  $\G_t=\U(3)$, $\G_q=\SO(5)\S^1$ and $\G_r=\SO(3)\U(2)$.
Moreover, $\H=\T^3$, $\G_{\ell _t}=\SO(3)\SO(2) \S^1$, and
$\K'=\SO(3)\lhd \G_{\ell _t}$.

Let $\SO(2)=\K'\cap \H \subset  \K'$. Consider the reduction
$(M^{\SO(2)}, \N(\SO(2)))$. It is once again a polar manifold with
the same section. For such a reduction, notice that: the face $\ell
_q$ has multiplicity $2$, the face $\ell _t$ is exceptional with
normal sphere $\Sph^0$, and $\G_q\cap \N(\SO(2))/\G_{\ell_t}\cap
\N(\SO(2))=\Sph^2$. Therefore, the action of $\N(\SO(2))$ is
reducible with fundamental chamber $rqq'$, where $q'$ is a
reflection image of $q$, and $\overline{rq}=\ell_t$ is of
exceptional orbit type. In particular, the multiplicities at $q'$ are
$(2, 2)$, hence the slice representation at $q'$ for the
$\N(\SO(2))$-action is again the adjoint representation of $\SO(5)$
on $\Bbb R^{10}$. This clearly implies that $q'$ is a fixed point.
On the other hand,  notice that $M^{\SO(2)}$ is orientable and hence
simply connected. Therefore, by Theorem 6.2 of \cite{FGT} we know that
$M^{\SO(2)}=\Sph^{10}$. Since Property (P) holds for $\SO(2)$ it follows from Remark \ref{rank3crit} that $\mathscr{C}(M,\G)$ is a building.
\end{proof}

\begin{rem}
We remark that in the above proof, the chamber system of
$(M^{\SO(2)}, \N(\SO(2)))$ is  a building of type $\A_1\times \CC_2$
but the one for $(M^{\SO(2)}, \N_0(\SO(2)))$ is not.
\end{rem}

\begin{prop} In the case of the multiplicities $(4, 4, 5)$, the chamber system
$\mathscr{C}(M;\G)$ is covered by a building, with the
isotropy representation of $\SO(14)/\U(7)$ as a linear model.
\end{prop}
\begin{proof}
By Lemma \ref{transv} we know that all isotropy groups are
connected. Note that $\bar\G_t= \Sp(3)$, and $\bar \G_q=\SU(5)$ or
$\U(5)$. By Lemma \ref{effectiveC3}, it is easy to see that:

$\bullet$ if $\G_t$ is semisimple, then, up to local isomorphism, $\G_t=\Sp(3)$, $\G_r= \Sp(2)\SU(3)$, $ \G_q=\SU(5)\Sp(1)$
and $\G_{\ell _t}=\SU(3)\Sp(1)^2$, where $\Sp(1)=\K_q$ is a subgroup
of $\G_t$.

$\bullet$ if $\G_t$ is not semisimple, then $\K_t=\S^1$, and all
isotropy groups data are the product of $\S^1$ with the
corresponding data above.

We now prove that $\G$ contains $\SU(7)$ as a normal subgroup.  By
Lemma \ref{WilkC3} the isotropy representations of $\G/\Sp(2)$ and
$\G/\SU(3)$ are both spherical, where $\Sp(2), \SU(3)$ are normal
factors of face isotropy groups. Hence, a normal factor $\L$ of $\G$
is either $\SO(n)$ or $\SU(n)$, by  Table B in [GWZ].
Moreover, the subgroup $\K_q\subset \G_t$ is contained in a block
subgroup $\SO(4)\subset \L$ (resp. a block subgroup $\SU(2)\subset
\L$) if $\L=\SO(n)$ (resp. $\L=\SU(n)$). Since $\N_0(\K_q)$ contains
$\G_q$, it follows that $n\ge 14$ (resp. $n\ge 7$) if $\L=\SO(n)$
(resp. $\SU(n)$). To rule out the former case, consider the fixed
point set $M^{\K_q}$ with the polar action of $\N_0(\K_q)$. It is
clearly a reducible cohomogeneity $2$ action with $q$ a vertex of angle
$\pi/4$. By the Dual Generation Lemma 7.2 of \cite{FGT} it follows that $\N_0(\K_q)$ is
either $\G_q$ (the fixed point case) or the product of $\SU(5)\lhd
\G_q$ with the face isotropy group opposite to $q$ in the reduction
$M^{\K_q}/\N_0(\K_q)$.  From this it is immediate that
$\L=\SU(7)$.

Note that if $\G_t$ is semisimple, or $\text{dim}(M)$ is even, then $\text{rank}\G\leq 6$,
by the Rank Lemma, and hence $\G=\SU(7)$. For the remaining case,
i.e., $\text{dim}(M)$ being odd and $\G_t=\S^1\cdot \Sp(3)$, we now
prove that $\G=\U(7)$, up to local isomorphism. Indeed, it is clear
that $\text{rank}\G =7$, and hence $\G=\SU(7)\cdot \L_2$, where
$\L_2$ is a rank $1$ group. It suffices to prove that $\L_2=\S^1$.
Let $\K'=\SU(3)\lhd \G_{\ell_t}$. It is clear that the projection
$p_2: \G\to \L_2$ is trivial, when restricted to either of $\Sp(3)\lhd
\G_t$ and $\K'\subset \G_q$. By primitivity \ref{prim}, $\G=\langle \G_t,
\G_{\ell_t}\rangle =\langle \G_t, \K'\rangle $. Therefore,
$p_2(\G_t)=\L_2$ and hence, $\L_2=\S^1$.

To complete the proof, we split into two cases, i.e, $\text{dim}(M)$
being even or odd. For the former, $\K_t=\S^1$ and $\G=\SU(7)$. It
is clear that $\G_t=\Sp(3)\S^1$ is a subgroup of $\U(6)\subset
\SU(7)$ and $\G_q=\SU(5)\Sp(1)\cdot\S^1$ is the normalizer
$\N(\Sp(1))$ in $\G$, where $\Sp(1)\lhd \G_{\ell _r}\subset \G_t$.
This forces all isotropy groups data to be the same as for the
linear cohomogeneity $2$ polar action on $\Bbb{CP}^{20}$ induced
from the isotropy representation of $\SO(14)/\U(7)$. Hence, in
particular, the chamber system $\mathscr{C}(M;\G)$ is covered by a
building. For the latter, $\G=\SU(7)$ or $\U(7)$ depending on
$\K_t=\{1\}$ or $\S^1$. The fixed point set $M^{\K'}$ is odd
dimensional, since the isotropy representation of $\K'$ is the
defining complex representation. Note that
$\N_0(\K')=\SU(4)\T^i\cdot \K'$, $i=1, 2$, and $M^{\K'}$ is
equivariantly diffeomorphic to $\Sph^{11}$ with a standard linear
cohomogeneity one action of type $\CC_2$. Hence, by  Lemma \ref{Criterion2},
$\mathscr{C}(M;\G)$ is a building.
\end{proof}


\providecommand{\bysame}{\leavevmode\hbox    
to3em{\hrulefill}\thinspace}                                            

\end{document}